\documentclass[11pt]{amsart}

\usepackage{epigamath}

\usepackage[english]{babel}

\numberwithin{equation}{section}

\usepackage{enumitem}
\usepackage{tikz-cd}
\usepackage{cleveref}
\usepackage{mathtools}
\usepackage[all]{xy}
\usepackage{xfrac}

\newtheorem{thm}[equation]{Theorem} %Theorem o Teorema a seconda
\newtheorem{claim}[equation]{Claim}
\newtheorem{prop}[equation]{Proposition}
\newtheorem{lemma}[equation]{Lemma}
\newtheorem{cor}[equation]{Corollary}

\theoremstyle{definition}

\newtheorem{question}[equation]{Question}

\theoremstyle{remark}
\newtheorem{example}[equation]{Example}
\newtheorem{rmk}[equation]{Remark}
\newtheorem{rmks}[equation]{Remarks}

\newcommand{\Q}{\mathbb Q}
\newcommand{\F}{\mathbb F}
\newcommand{\Z}{\mathbb Z}
\newcommand{\Spec}{\operatorname{Spec}}

\newcommand{\Char}{\operatorname{char}}

\newcommand{\Mat}{\operatorname{M}}

\newcommand{\GL}{\operatorname{GL}}
\newcommand{\G}{\mathbb G}
\renewcommand{\P}{\mathbb P}

\newcommand{\Sym}{\operatorname{S}}

\newcommand{\ed}{\operatorname{ed}}

\newcommand{\trdeg}{\operatorname{trdeg}}

\newcommand{\mc}[1]{\mathcal{#1}}
\newcommand{\cl}{\overline}
\newcommand{\set}[1]{\left\{#1\right\}}
\newcommand{\on}[1]{\operatorname{#1}}

\newcommand{\frp}{\mathfrak{p}}

\EpigaVolumeYear{6}{2022} \EpigaArticleNr{21} \ReceivedOn{January 4, 2022}
\InFinalFormOn{May 8, 2022}
\AcceptedOn{July 29, 2022}

\title{The behavior of essential dimension under specialization}
\titlemark{Essential dimension under specialization}

\author{Zinovy Reichstein}
\address{Department of Mathematics,
        University of British Columbia,
        Vancouver, BC V6T 1Z2, Canada}
\email{reichst@math.ubc.ca}
\author{Federico Scavia}
\address{Department of Mathematics,
        University of California,
        Los Angeles, CA 90095-1555, USA}
\email{scavia@math.ucla.edu}

\authormark{Z. Reichstein and F. Scavia}

\AbstractInEnglish{Let $A$ be a discrete valuation ring with generic point $\eta$ and closed point $s$.
        We show that in a family of torsors over $\Spec(A)$, the essential dimension of the torsor above $s$
        is less than or equal to the essential dimension of the torsor above $\eta$. We give two applications of this result, one in mixed
        characteristic, the other in equal characteristic.}



\acknowledgement{Zinovy Reichstein was partially supported by
        National Sciences and Engineering Research Council of
        Canada Discovery grant 253424-2017. Federico Scavia was partially supported by a graduate fellowship from the University of British Columbia.}

\begin{document}

%%%%%%%%%%%%%%%%%%%%%%%%%%%%%%%
% Title page
%%%%%%%%%%%%%%%%%%%%%%%%%%%%%%%

%\removeabove{}
%\removebetween{}
%\removebelow{}

\maketitle

\begin{prelims}

\DisplayAbstractInEnglish

\bigskip

\DisplayKeyWords

\medskip

\DisplayMSCclass

%\bigskip

%\languagesection{Fran\c{c}ais}

%\bigskip

%\DisplayTitleInFrench

%\medskip

%\DisplayAbstractInFrench

\end{prelims}

%%%%%%%%%%%%%%%%%%%%%
% Table of Contents
%%%%%%%%%%%%%%%%%%%%%

\newpage

\setcounter{tocdepth}{1}

\tableofcontents

%%%%%%%%%%%%%%%%%%%%%
% Content begins here
%%%%%%%%%%%%%%%%%%%%%

        \section{Introduction}
        
        Let $k$ be a field, $G$ be a linear algebraic group over $k$, $K/k$ be a field extension and $\tau \colon \mathcal{T} \to \Spec(K)$ be a $G$-torsor. We say that $\tau$ (or $\mathcal{T}$) descends to an intermediate subfield $k \subset K_0 \subset K$ if there exists a $G$-torsor $\tau_0 \colon \mathcal{T}_0 \to \Spec(K_0)$ such that $\tau_0$ is obtained from $\tau$ by a pull-back diagram
        \begin{equation} \label{e.compression}\begin{gathered}
        \xymatrix{   \mathcal{T} \ar@{->}[d]_{\tau}  \ar@{->}[r] & \mathcal{T}_0 \ar@{->}[d]^{\tau_0}   \\          
                               \Spec(K) \ar@{->}[r]    & \Spec(K_0) .} \end{gathered}
        \end{equation}
     The essential dimension $\ed_k(\tau)$ is the minimal transcendence degree $\trdeg_k(K_0)$ such that
    %  $\pi$ 
     $\tau$ descends to $K_0$. 
    Essential dimension of
    torsors has been much studied; for an overview see~\cite{reichstein2010essential} or~\cite{merkurjev2013essential}.
    In this paper we will investigate how $\ed_k(\tau)$ behaves as we deform $\tau$. Our main theorem shows that 
    under relatively mild assumptions, $\ed_k(\tau)$ does not increase under specialization.  We will sometimes 
    write $\ed_k(\mathcal{T})$ or $\ed_k([\tau])$ in place of $\ed_k(\tau)$, where $[\tau]$ is the class of $\tau$ in $H^1(K, G)$. 
    
    Following \cite{brosnan2018essential}, we will refer to a finite group $S$ as being ``tame'' at a prime $p$ if $|S|$ 
    is not divisible by $p$ and ``weakly tame'' at $p$ if $S$ does not have a non-trivial normal $p$-subgroup. 
        By definition every finite group is tame at $p = 0$.
        
        \begin{thm}\label{puiseux}
                Let $A$ be a complete discrete valuation ring with maximal ideal $\mathfrak{m}$, fraction field $k$ and residue field $k_0$.
                Set $p := \on{char}(k_0) \geqslant 0$ and let $G$ be a smooth affine group scheme over $A$, satisfying 
                one of the conditions (i), (ii) or (iii) below. Let $R\supset A$ be a complete discrete valuation ring with fraction field $K\supset k$ and residue field $K_0\supset k_0$, and assume that $\mathfrak{m}$ is contained in the maximal ideal of $R$. Then for every $\alpha\in H^1(R,G)$ we have \[\ed_{k_0}(\alpha_{K_0})\leqslant \ed_{k}(\alpha_{K}).\]
                Furthermore, if $A=k_0[[t]]$ and $G_{A}$ is defined over $k_0$, then the above inequality is an equality.
                
                        \begin{enumerate}[label=(\roman*)]
                        \item $p = 0$, and there exist a section $\sigma:k_0\to A$ of the projection $A\to k_0$, and a $k_0$-group $H$ such that $G\simeq \sigma^*H$.
                        \item The neutral component $G^{\circ}$ is reductive, $G/G^{\circ}$ is $A$-finite, and there exists a finite subgroup $S\subset G(A)$ such that $S$ is tame at $p$ and
                        for every field $L$ containing $k$ the natural map $H^1(L,S)\to H^1(L,G)$ is surjective.
                        \item $G = S_{A}$, where $S$ is an abstract finite group which is weakly tame at $p$.
                        % is the constant $A$-group scheme associated to a weakly tame finite abstract group.
                \end{enumerate}
        \end{thm}

        The assumption that $A$ and $R$ are complete may be dropped; see \Cref{cor.not-complete}. In the case, where
        $G$ is a finite group and $p = 0$, the inequality $\ed_{k_0}(\alpha_{K_0})\leqslant \ed_{k}(\alpha_{K})$
        was noted in~\cite[Remark 6.3]{fakhruddin2021finite}.
        A version of \Cref{puiseux} for essential dimension at a prime will be proved in the Appendix; see \Cref{puiseux'}. 
        
        It is natural to ask whether or not conditions (ii) and (iii) can be replaced by a single (weaker) assumption. 

\begin{question} \label{q.tame}
Does~\Cref{puiseux} remain valid if the finite group $S$ in part (ii) is only assumed to be weakly tame, rather than tame?
\end{question}

            We have not been able to answer this question; our proof of~\Cref{puiseux} uses an entirely different argument in case (iii),
            compared to cases (i) and (ii). 
           Note however, that \Cref{puiseux}(ii) and (iii) both fail in the case where $p > 0$ and $G$ is a finite discrete $p$-group;
           see~\Cref{lem.versality3}. Some assumptions on $\on{char}(k_0)$ are thus necessary. Note also that~\Cref{q.tame} has a positive answer
           if essential dimension is replaced by essential dimension at a prime $q$, different from $p$; see~\Cref{puiseux'}.

We will give two applications of~\Cref{puiseux}. Additional applications of \Cref{puiseux} can be found in the companion paper~\cite{gabber3}.

For our first application, recall that the essential dimension $\ed_k(G)$ of an algebraic group $G$ defined 
over $k$ is the supremum of $\ed_k(\alpha)$, as $K$ ranges over field extensions of $k$ and $\alpha$ ranges over $H^1(K, G)$. It is shown in \cite{brosnan2018essential} that if
$G$ is an abstract finite group which is weakly tame at a prime $p > 0$, then 
\begin{equation} \label{e.brv}
\ed_k(G) \geqslant \ed_{k_0} (G), 
\end{equation}
where $k$ is any field of characteristic $0$ and $k_0$ is any field of characteristic $p$ containing the algebraic closure of the prime field $\F_p$. \Cref{thm.split-red} below partially 
extends this inequality to split reductive groups.~\footnote{For a recent generalization of the inequality~\eqref{e.brv} in a different (stack-theoretic) direction, see~\cite{bresciani-vistoli}.}
        
        \begin{thm}\label{thm.split-red}
                Let $G$ be a split reductive group scheme (not necessarily connected) of rank $r \geqslant 0$ defined over $\Z$. Denote the Weyl group of $G$ by $W$. 
                Then \[ \ed_k(G_k)\geqslant\ed_{k_0}(G_{k_0}) \]
                for any field $k$ of characteristic zero and any field $k_0$ of characteristic $p > 0$, as long $p$ does not divide $2^r |W|$
                and $k_0$ contains the algebraic closure of $\F_p$. 
        \end{thm}
        
Note that in the case where $G$ is a finite constant group (and thus $r = 0$ and $W = G$), we only recover Inequality~\eqref{e.brv} 
in the case where $G$ is tame at $p$. A positive answer to \Cref{q.tame} would imply that \Cref{thm.split-red} remains valid when $W$ 
is only assumed to be weakly tame at $p = \Char(k_0)$, as long as $p$ is odd. 

For our second application of \Cref{puiseux}, we briefly recall the definition of essential dimension of a $G$-variety. 
Once again, let $G$ be an algebraic group defined over $k$. By a $G$-variety we shall mean a separated reduced $k$-scheme of finite type endowed with a $G$-action over $k$. We will say that the $G$-variety $Y$ is primitive if $Y \neq \emptyset$ and $G(\cl{k})$ transitively permutes the irreducible components of $Y_{\cl{k}}:=Y \times_k\cl{k}$. We will say that the $G$-variety $Y$ (or equivalently, the $G$-action on $Y$) is generically free
if there exists a dense open subscheme $U \subset Y$ such that for every $u\in U$ the scheme-theoretic stabilizer $G_u$ of $u$ is trivial. 

Now suppose that $Y$ is a generically free primitive $G$-variety defined over $k$. By a $G$-compression of $Y$ we will mean 
a dominant $G$-equivariant rational map $Y \dashrightarrow X$, defined over $k$,
where the $G$-action on $X$ is again generically free and primitive. The essential dimension of $Y$, denoted by $\ed_k(Y;G)$, 
is defined as the minimal value of $\dim(X) - \dim(G)$, where the minimum is taken over all $G$-compressions $Y \dashrightarrow X$.
This notion is closely related to that of essential dimension of a torsor; see~\Cref{sect.ed-variety}. If the reference to $G$ is clear from the context, we will often write $\ed_k(Y)$ in place of $\ed_k(Y;G)$.
        
        \begin{thm} \label{thm.map} Let $k$ be a field of characteristic $p\geqslant 0$,  $G$ be a linear algebraic group defined over $k$, satisfying one of the conditions (1) -- (4) below.
        Let $X$, $Y$ be primitive generically free $G$-varieties defined over $k$. Assume that $X$ is %quasi-projective and 
        smooth.         If there exists a $G$-equivariant rational map $f \colon Y \dasharrow X$, then $\ed_k(X) \geqslant \ed_k(Y)$.
                
        \begin{enumerate}
            \item $p=0$.
            \item $p > 0$, $G$ is split reductive, % $p$ does not divide $2^r n$, 
            and $k$ contains a primitive root of unity of degree $2 n^2$.
            % $T$ is a split torus, and $T[n^2]$ is discrete. 
            % FS: here we don't need T
            \item $p >0$, $G$ is connected reductive, % $p$ does not divide $n$, 
            and there exists a maximal torus $T$ of $G$ such that
            $T[n^2]$ is discrete.
            %FS: there exists a maximal $k$-torus $T$ such that
            \item $G$ is a finite discrete group, weakly tame at $p$.
        \end{enumerate} 

\noindent
        Here in parts (2) and (3), $n = |W|$ denotes the order of the Weyl group $W$ of $G_{\overline{k}}$. 
        \end{thm}

If $f$ is dominant, then \Cref{thm.map} is obvious from the definition, since any compression of $X$ can be composed with $f$.
The key point here is that $f$ is allowed to be arbitrary. In particular, we do not assume that the $G$-action on
the image of $f$ is generically free. The idea of the proof is to use~\Cref{puiseux} to deform $f$.
The assumption that $X$ be smooth may not be dropped; see \Cref{rem.smoothness}.

\subsection*{Acknowledgements} We are grateful to the referee for constructive comments and to Angelo Vistoli for his help with Lemma~\ref{large-algebraic-space}. The second-named author thanks K{\fontencoding{T1}\selectfont\k{e}}stutis \v{C}esnavi\v{c}ius and the D\'epartement de Math\'ematiques d'Orsay (Universit\'e Paris-Saclay) for hospitality during Summer 2021, and the Institut des Hautes \'Etudes Scientifiques for hospitality in the Fall 2021.

\section{Preliminaries}\label{prelim}

\subsection{Dependence on the base field}
\label{sect.base-field}

\begin{lemma} \label{lem.prel1}
Let $G$ be a linear algebraic group over $k$, $K/k$ be a field extension, and $\tau \colon \mathcal{T} \to \Spec(K)$ be a $G$-torsor.

\begin{enumerate}[label=(\alph*)]
\item If $K'/K$ is a field extension, then 
$\ed_k(\tau) \geqslant \ed_k(\tau_{K'})$. Here $\tau_{K'} \colon \mathcal{T}_{K'} \to \Spec(K')$ is the $G$-torsor obtained from $\tau$ 
by base-change via the natural map $\Spec(K') \to \Spec(K)$.
\item If $k \subset l\subset K$, then $\ed_{k}(\tau) \geqslant \ed_{l}(\tau)$.
\end{enumerate}
\end{lemma}

\begin{proof}  Consider Diagram~\eqref{e.compression} with smallest possible value of $\trdeg_k(K_0)$, that is,
$\trdeg_k(K_0) = \ed_k(\tau)$. 

(a) Composing with the natural projection $\mathcal{T}_{K'} \to \mathcal{T}$, we obtain a Cartesian diagram of $G$-torsors
\[ \xymatrix{  \mathcal{T}_{K'} \ar@{->}[d]_{\tau_{K'}}  \ar@{->}[r] &  \mathcal{T} \ar@{->}[d]^{\tau}  \ar@{->}[r] & \mathcal{T}_0 \ar@{->}[d]^{\tau_0}   \\          
                        \Spec(K') \ar@{->}[r]     &        \Spec(K) \ar@{->}[r]     & \Spec(K_0) ,} \]
which shows that $\ed_k(\tau_{K'}) \leqslant \trdeg_k(K_0) = \ed_k(\tau)$.  

(b) Choose an intermediate field $k \subset K_0 \subset K$ such that $\trdeg_{k}(K_0) = \ed_{k}(\tau)$.
Let $K_1$ be the subfield of $K$ generated by $l$ and $K_0$. Then $\tau$ also descends to $K_1$. Thus
\[ \ed_l(\tau) \leqslant \trdeg_l(K_1) \leqslant \trdeg_{k}(K_0) = \ed_{k}(\tau) .\qedhere \]
\end{proof} 

\subsection{Essential dimension of a $G$-variety}
\label{sect.ed-variety}

Let $G$ be an algebraic group defined over $k$ and let $Y$ be a generically free primitive $G$-variety. These terms 
are defined in the paragraph preceding \Cref{thm.map}, where one can also find the definition of a compression
$Y \dasharrow X$ and of $\ed_k(Y)$.

\begin{lemma} \label{lem.prel1.5} Let $Y$ be a generically free primitive $G$-variety defined over $k$, and $k'/k$ be a field extension.
Then 
\begin{enumerate}[label=(\alph*)]
\item $\ed_k(Y) \geqslant \ed_{k'} (Y_{k'})$.
\item There exists an intermediate field $k \subset l \subset k'$ such that $l$ is finitely generated over $k$ and $\ed_l(Y_l) = \ed_{k'}(Y_{k'})$.
\item If $k$ is algebraically closed, then $\ed_k(Y) = \ed_{k'} (Y_{k'})$.
\end{enumerate}
\end{lemma} 

Note that if $Y$ is a generically free primitive $G$-variety, then $Y_{k'}$ is a generically free primitive $G_{k'}$-variety.

\begin{proof} (a) The inequality $\ed_k(Y) \geqslant \ed_{k'} (Y_{k'})$ follows directly from the definition of essential dimension, since
every $G$-compression $Y \dasharrow X$ gives rise to a $G_{k'}$-compression $Y_{k'} \dasharrow X_{k'}$ by base-change. 

(b) By (a), it suffices to find $l$ so that $\ed_l(Y_l) \leqslant \ed_{k'}(Y_{k'})$. Let
$f' \colon Y_{k'} \dashrightarrow X'$ be a $G_{k'}$-compression over $k'$ such that $X'$ is a generically free primitive $G_{k'}$-variety and $\dim_{k'}(X')= \ed_{k'}(Y_{k'})$.  
Then there exists an intermediate field $k \subset l \subset k'$ such that $l$ is finitely generated over $k$ and $f'$ descends to a $G_l$-compression $f \colon Y_l \to X$, where $X$ is a generically free $G_l$-variety and $X$ is a generically free primitive $G_l$. Therefore
\[\ed_l(Y_l)\leqslant \dim_l(X)=\dim_{k'}(X')=\ed_{k'}(Y_{k'}),\]
as desired. 

(c) In view of part (b), we may assume without loss of generality that $k'$ is finitely generated over $k$. By part (a), it suffices to show that 
$\ed_k(Y) \leqslant \ed_{k'} (Y_{k'})$.
Let $f \colon Y_{k'}\dasharrow X$ be a $G$-compression over $l$ such that $\dim(X)=\dim (G)+\ed_{k'}(Y_{k'})$.
Then $f$ is actually defined over some $k$-variety $U$ whose function field is $k'$:
\[
                \xymatrix{Y \times_k U \; \ar@{->}[dr]_{{\rm pr}_2} \ar@{-->}[rr]^{F} &  & \mathcal{X} \ar@{->}[dl]^{\pi}    \cr
                         & U .&   } 
 \]
In other words, there exists a surjective $G$-equivariant morphism 
$\pi \colon \mathcal{X} \to U$ whose generic fiber is $X$, and a $G$-equivariant dominant rational map
$F \colon Y \times_k U \dasharrow \mathcal{X}$ over $U$ 
whose generic fiber is $f$. Since $k$ is algebraically closed, $U(k)$ is dense in $U$. 
If $u \in U(k)$ is a $k$-point in general position in $U$, then the morphism $F_u \colon Y \dasharrow \mathcal{X}_u$ is a $G$-compression of $Y$, 
and $\dim(\mathcal{X}_u)=\dim (X)$. We conclude that \[\ed_k(Y) \leqslant \dim(\mathcal{X}_u)-\dim(G)=\dim (X) -\dim G = \ed_{k'}(Y_{k'}), \] 
as desired.
\end{proof}

    Recall that in the Introduction we defined the essential dimension for both $G$-torsors and $G$-varieties.
    These two notions are closely related 
    % to each other 
    in the following way. Let $Y$ be a generically free primitive
    $G$-variety. After passing to a $G$-invariant open subvariety of $Y$, we may assume that $Y$ is the total space 
    of a $G$-torsor $\tau \colon Y \to B$, where
        $B$ is irreducible with function field $K = k(B) = k(Y)^G$; see~\cite[Theorem 4.7]{berhuy2003essential}.
        Note that $k(Y)$ is a field if $Y$ is irreducible and a direct product of fields in general; however, $K$ is always a field,
        as long as $Y$ is primitive.
        Pulling back to the generic point $\Spec(K) \to B$ of $B$, we obtain a $G$-torsor $\tau_Y \colon \mathcal{T}_Y \to \Spec(K)$.
        Now an easy spreading out argument shows that  
        \begin{equation} \label{e.torsor-variety} \ed_k(Y) = \ed_k(\tau_Y); 
        \end{equation}
        see~\cite[Lemma~3.9]{merkurjev2013essential}.
                
        Recall from the Introduction that $\ed_k(G)$ is defined as the maximal value of $\ed_k(\tau)$, 
        where the maximum is taken over all field extensions of $K/k$ and all $G$-torsors $\tau \colon \mathcal{T} \to \Spec(K)$.
        
\begin{lemma} \label{lem.prel2} Let $G$ be a linear algebraic group defined over a field $k$ and $k'/k$ be a field extension.

\begin{enumerate}[label=(\alph*)]
\item  Let $V$ be a generically free linear representation of $G$ defined over $k$. Then $\ed_k(V) = \ed_k(G)$.
\item $\ed_k(G)$ is the maximal value of $\ed_k(Y)$, where $Y$ ranges over generically free primitive $G$-varieties. 
\item $\ed_k(G) \geqslant \ed_{k'}(G_{k'})$. 
\item {\rm (}\cite[Proposition 2.14]{brosnan2007essential}, \cite[Example 4.10]{tossici2017essential}{\rm)}
Moreover, $\ed_k(G) = \ed_{k'}(G_{k'})$ if $k$ is algebraically closed.
\item {\rm (}\cite[Lemma 4.8]{tossici2017essential}{\rm)}
There exists an intermediate field $k \subset l \subset k'$ such that $l$ is finitely generated over $k$ and
$\ed_l(G_l) = \ed_{k'}(G_{k'})$.
\end{enumerate}
\end{lemma} 

\begin{proof} (a) This follows from the fact that the $G$-action on $V$ is versal, which is, in turn, a consequence of Hilbert's Theorem 90. See~\cite[Propositions 3.10 and 3.11]{merkurjev2013essential} for details.

(b) For any generically free primitive $G$-variety $Y$, $\ed_k(Y) = \ed_k(\tau_Y) \leqslant \ed_k(G)$. On the other hand,
note that since $G$ is a linear algebraic group, there exists a closed embedding $G \hookrightarrow \GL_n$ for some $n \geqslant 1$.
Let $G$ act on the space $V = \Mat_n$ of $n \times n$ matrices by left multiplication via this embedding. Now part (a) tells us that
$\ed_k(V) = \ed_k(G)$, and part (b) follows.

(c--d) By~\Cref{lem.prel1.5}, $\ed(V; G) \geqslant \ed(V_{k'}; G_{k'})$; moreover, equality holds if $k$ is algebraically closed.
The desired conclusions now follows from part (a).

(e) Let $V$ be a generically free linear representation of $G$ defined over $k$. 
By~\Cref{lem.prel1.5}(b), there exists an intermediate extension $k \subset l \subset k'$
such that $l$ is finitely generated over $k$ and $\ed_{l}(V_l) = \ed_{k'}(V_{k'})$.
By part (a), the left hand side of this equality is $\ed_l(G_l)$ and the right hand side is
$\ed_{k'}(G_{k'})$.  
\end{proof}

\subsection{Essential dimension at a prime}
\label{sect.ed-at-p} 

       Let $G$ be a linear algebraic group defined over a field $k$ and let $q$ be a prime integer. Essential 
       dimension at $q$ for $G$-torsors, $G$-varieties, and $G$ itself, is defined in a way that parallels 
       the definitions of essential dimension for these objects. We recall these definitions below.

                A field $K$ is called $q$-closed if every finite field extension $L/K$ is of degree $[L:K] = q^r$ for some integer $r \geqslant 0$. For every field $K$,  there exists a unique algebraic extension $K^{(q)}/K$ such that $K^{(q)}$ is $q$-closed and the degree $[L: K]$ of every every finite subextension $K \subset L \subset K^{(q)}$ is prime to $q$. The field $K^{(q)}$ is called a $q$-closure of $K$. It is the fixed field of a $q$-Sylow subgroup of the absolute Galois group of $K$.

        Let $k$ be a field, and let $G$ be an algebraic $k$-group. If $K/k$ is a field extension and $\tau \colon \mathcal{T} \to \Spec(K)$ is a $G$-torsor, 
        then $\ed_{k,q}(\tau)$ is the minimal value of $\ed_k(\tau_L)$, where $L$ ranges over 
        the finite field extensions of $K$ whose degree $[L:K]$ is prime to $q$. Equivalently, 
        \begin{equation} \label{e.ed-at-p2} \ed_{k,q}(\tau) = \ed_k(\tau_{K^{(q)}}),
        \end{equation}
        where $K^{(q)}$ is the $q$-closure of $K$. 
        
        Recall that a correspondence $X \rightsquigarrow Z$ between $G$-varieties $X$ and $Z$ of 
        degree $d$ is a diagram of $G$-equivariant rational maps of $G$-varieties the form
        \begin{equation} \label{e.correspondence}
                \xymatrix{X' \; \ar@{-->}[d]_{\text{\tiny degree $d$}} \ar@{-->}[dr] & \cr
                        X  & Z  \, ,  } 
        \end{equation}
        where the vertical map is dominant of degree $d$. We say that $X \rightsquigarrow Z$
        is dominant if the rational map $X' \dasharrow Z$ in the above
        diagram is so.  Dominant correspondences may be composed in an evident way.
        Note that sometimes the term ``rational correspondence'' is used in place of correspondence, to indicate that
        the maps in Diagram~\eqref{e.correspondence} are rational maps. All correspondences used in this paper will be rational in this sense.
        For notational simplicity we will use the term ``correspondence'' throughout.
        
        Now let $X$ be a generically free primitive $G$-variety over $k$. The essential dimension of $\ed_{k,q}(X)$ of $X$ at $q$ is 
        the minimal value of $\dim(Z) - \dim(G)$, where the minimum is taken over all $G$-equivariant dominant correspondences $X \rightsquigarrow Z$ of degree prime to $q$. Equivalently,
        \begin{equation} \label{e.ed-at-p1}
        \begin{split}
                \ed_{k,q}(X) = \ &\text{maximal value of } \ed_{k,q}(X'), \ \text{where the maximum is taken over all $G$-equivariant} \\
               &\text{dominant rational covers $X' \dasharrow X$ of degree prime to $q$.}
                \end{split}
        \end{equation}
        If $\tau_X$ is the $G$-torsor over $k(X)^G$ associated to $X$, as in the previous section, then one readily checks that
        \begin{equation} \label{e.torsor-variety-at-p}
        \ed_{k, q}(X) = \ed_{k, q}(\tau_X).
        \end{equation}
        
        The essential dimension $\ed_{k,q}(G)$ of $G$ at $q$ is the maximal value of $\ed_{k,q}(\tau)$, as $K$ ranges over 
        fields containing $k$ and $\mathcal{T} \to \Spec(K)$ ranges over $G$-torsors over $\Spec(K)$. Equivalently,
        $\ed_{k, q}(G)$ is the maximal value of $\ed_{k,q}(X)$, as $X$ ranges over the generically free primitive $G$-varieties.
        
        \section{Torsors over complete discrete valuation rings} 
        \label{sect.laurent}
        
        Let $R$ be a complete discrete valuation ring, $F$ its fraction field, $F_0$ its residue field, and $\cl{s}$ and $\cl{\eta}$ geometric points lying over the closed and generic point of $\Spec R$, respectively. We have a group homomorphism
        \[\on{sp}:\on{Gal}(F)\to \on{Gal}(F_0),\] called specialization homomorphism. Under the identifications $\on{Gal}(F)=\pi_1(\Spec F,\cl{\eta})$ and $\on{Gal}(F_0)=\pi_1(\Spec F_0,\cl{s})$, the homomorphism $\on{sp}$ is the composition 
        \begin{equation}\label{specialization}\pi_1(\Spec F,\cl{\eta})\to \pi_1(\Spec R,\cl{\eta})\xrightarrow{\sim} \pi_1(\Spec F_0,\cl{s}),\end{equation}
        where the map on the left is induced by the inclusion $\Spec F\hookrightarrow \Spec R$, and the isomorphism on the right by specialization of finite \'etale covers; see {\it e.g.} \cite[Tag 0BUP]{stacks-project}. The homomorphism $\on{sp}$ may also be defined Galois-theoretically; see \cite[Section 7.1]{serre-ci}.

        \begin{lemma}\label{system} Let $K$ be a field, $\pi\in K^{\times}$, and fix an algebraic closure $\cl{K}$ of $K$. Then 
        
        (a) there exists a system $\set{\pi^{1/m}}_{m\geqslant1}$ of roots of $\pi$ in $\cl{K}$ such that $(\pi^{1/m_1 m_2})^{m_2}=\pi^{1/m_1}$ for any $m_1, m_2 \geqslant1$.
        
        (b) Moreover, suppose $\alpha \in \cl{K}$ satisfies $\alpha^n=\pi$ for some $n \geqslant 1$. Then the system $\set{\pi^{1/m}}_{m\geqslant 1}$ in part (a) 
        can be chosen so that $\pi^{1/n} = \alpha$.
        \end{lemma}
        
        \begin{proof} It suffices to prove part (b). Once (b) is established, we deduce (a) by setting $n = 1$ and $\alpha = \pi$.
        
        To prove (b), choose $\beta \in \cl{K}$ such that $\beta^{(n-1)!}= \alpha$. In particular, $\beta^{n!}= \pi$. 
        For every integer $e\geqslant d$, we define $\pi^{1/e!}\in \cl{K}$ recursively as follows: 
        $\pi^{1/n!}:= \beta$, and if $e> n$ then $\pi^{1/e!}$ is defined as an arbitrary $e^{\mathrm{th}}$ root of $\pi^{1/(e-1)!}$. 
        Now we define $\pi^{1/m}$ for an arbitrary integer $m\geqslant 1$ as follows. 
        Choose $e \geqslant n$ such that $m$ divides $e!$, and set $\pi^{1/m}:=(\pi^{1/e!})^{e!/m}$. 
        It is immediate to check that $\pi^{1/m}$ does not depend on the choice of $e$, that 
        $(\pi^{1/m_1 m_2})^{m_2}=\pi^{1/m_1}$ for any $m_1, m_2 \geqslant 1$ and that 
        $\pi^{1/n}=(\pi^{1/n!})^{(n-1)!}= \beta^{(n-1)!}= \alpha$.
        \end{proof}

        \begin{lemma}\label{big-diag}
                Let $R$ be a complete discrete valuation ring with fraction field $F$ and residue field $F_0$. 
                Set $p = \Char(F_0) \geqslant 0$. Choose a uniformizing parameter $\pi$ for $R$ and
                a system $\set{\pi^{1/n}}_{n\geqslant 1}$ of roots of $\pi$ in $\cl{F}$ such that $(\pi^{1/mn})^m=\pi^{1/n}$ for all $m,n\geqslant 1$.
                Let $F(\pi^{1/\infty})$ be the subfield of $\cl{F}$ obtained by adjoining all the $\pi^{1/n}$ to $F$. Then we have a commutative diagram
                \[
                \begin{tikzcd}
                        & 1 \arrow[d] & \\
                        1 \arrow[r]  & P \arrow[r] \arrow[d] & \on{Gal}\left(F(\pi^{1/\infty})\right)\arrow[r,"\on{sp}'"] \arrow[d,hook] & \on{Gal}(F_0) \arrow[r] \arrow[d,equal] & 1 \\
                        1 \arrow[r] & \on{Gal}(F_{\on{nr}}) \arrow[r] \arrow[d] & \on{Gal}(F)\arrow[r,"\on{sp}"] \ & \on{Gal}(F_0) \arrow[r] & 1, \\
                        & \Q/\Z[1/p] \arrow[d]  \\
                        & 1
                \end{tikzcd}
                \]
                where the rows and columns are exact and $P$ is a pro-$p$-group. $($If $p=0$, then $P$ is trivial and $\Q/\Z[1/p]$ should be interpreted as $\Q/\Z$.$)$ Moreover, the surjective homomorphism $\on{sp}'$ admits a section.
        \end{lemma}
        
        \begin{proof}
                The bottom horizontal exact sequence comes from the definition of the maximal unramified extension $F_{\on{nr}}$; see  \cite[\S~II.4.3, Exercise~2(c)]{serre1997galois}.
                
                Define $P:=\on{Gal}(F_{\on{mod}})$, where $F_{\on{mod}}=\cup_{(n,p)=1}F_{\on{nr}}(\pi^{1/n})$ is the maximal tamely ramified extension of $F$; see \cite[\S~II.4.3, Exercise~2(a)]{serre1997galois}, where it is asserted that $F_{\on{mod}}/F_{\on{nr}}$ is Galois with Galois group isomorphic to $\Q/\Z[1/p]$. The profinite group $P$ is a pro-$p$-group by \cite[\S~II.4.3, Exercise~2(b)]{serre1997galois}. This yields the vertical short exact sequence.  
                
                Let $F(\pi^{1/\infty})':=\bigcup_{(n,p)=1}F(\pi^{1/n})$. Then $F(\pi^{1/\infty})/F(\pi^{1/\infty})'$ is purely inseparable (in particular, $F(\pi^{1/\infty})=F(\pi^{1/\infty})'$ if $p=0$) and $F(\pi^{1/\infty})'/F$ is Galois with Galois group $\Q/\Z[1/p]$ (if $p>0$) or $\Q/\Z$ (if $p=0$).
                
                It remains to construct the top horizontal exact sequence. The map $\on{sp}'$ is defined as the restriction of $\on{sp}$. We claim that $\on{sp}'$ is surjective. Let $\gamma\in \on{Gal}(F_0)$: we will show that $\gamma$ belongs to the image of $\on{sp}'$. Let $L_0/F_0$ be a finite Galois extension such that $L_0$ is invariant under $\gamma$, and $\gamma$ restricts to an element of $\on{Gal}(L_0/F_0)$. Let $F\subset L\subset \cl{F}$, where $L/F$ is the unramified extension corresponding to $L_0/F_0$. By construction $L/F$ induces $L_0/F_0$ by passing to residue fields, and the homomorphism $\on{sp}_{L/F}:\on{Gal}(L/F)\to \on{Gal}(L_0/F_0)$ induced by $\on{sp}$ is an isomorphism. Let $L(\pi^{1/\infty})$ be the subfield of $\cl{F}$ generated by $L$ and $\set{\pi^{1/n}}_{n\geqslant1}$. We have an $F(\pi^{1/\infty})$-algebra homomorphism $f \colon L\otimes_FF(\pi^{1/\infty})\to L(\pi^{1/\infty})$ given by $\lambda\otimes z\mapsto \lambda z$.
                
                We would like to show that $f$ is an $F(\pi^{1/\infty})$-algebra isomorphism.         
                First we note that the image of $f$ contains both $L$ and $\pi^{1/n}$ for every $n \geqslant 1$. Hence, $f$ is surjective.
            
                %% show that $f$ is surjective. %It suffices to show that, 
                %% for all $n\geqslant 1$, the homomorphism of $F(\pi^{1/n})$-algebras 
                %% $L\otimes_FF(\pi^{1/n})\to L(\pi^{1/n})$ is surjective.              
                %% Indeed, if $z \in L(\pi^{1/\infty})$, then $z \in L(\pi^{1/n})$ for some $n %% \geqslant 1$, and so there exist $\lambda_0,\dots,\lambda_{n-1}\in L$ 
                %% such that
                %% \[z=\sum_{i=0}^{n-1}\lambda_i\pi^{i/n}=
                %% f\left(\sum_{i=0}^{n 1}\lambda_i\otimes\pi^{i/n}\right).\]
                %$z=\sum_{i \geqslant N} \lambda_i\pi^{i/n}$, for some $n \geqslant 1$, 
                %% some $N \in \Z$, 
                %and some $\lambda_i\in L$.
                % and $N$ are integers. 
                %Write $\lambda_i=\sum_{j=1}^m c_{ij}a_j$ for some $c_{i1},\ldots,c_{im} \in F^{\times}$. Then
                %\[z=\sum_{j=1}^ma_j\left(\sum_{i\geqslant N}c_{ij}\pi^{i/n}\right)=f\left(\sum_{j=1}^ma_j\otimes \left(\sum_{i\geqslant N}c_{ij}\pi^{i/n}\right)\right).\]
                %This proves that $f$ is surjective.                 
                It remains to show that $f$ is injective, i.e., that 
                $f \colon L\otimes_FF(\pi^{1/n})\to L(\pi^{1/n})$ is injective for every $n$.
                To prove this, it suffices to show that 
                $1, \pi^{1/n},  \ldots, \pi^{(n-1)/n}$ are linearly 
                independent over $L$. Indeed, suppose
                \begin{equation} \label{e.lin-independent}
                 l_0 + l_1 \pi^{1/n} +  \ldots + l_{n-1} \pi^{(n-1)/n} = 0 
                \end{equation}
                for some $l_0, l_1, \ldots, l_{n-1} \in L$.
                Since $L/F$ is unramified,
                the given valuation $\nu \colon F^* \to \mathbb Z$ lifts to valuations
                $L^* \to \mathbb{Z}$ and $L(\pi^{1/n})^* \to \dfrac{1}{n} \mathbb Z$
                which, by abuse of notation, we will also denote by $\nu$. Now observe that
                the terms on the left hand side of~\eqref{e.lin-independent} all have different valuations, as $\nu(l_i \pi^{i/n})$ is of the form $i/n$ plus an integer. 
                % nu(l_i)$, where
                % $\nu(l_i) \in \mathbb Z$. 
                Thus the only way the sum in~\eqref{e.lin-independent}
                can be $0$ is if each term is $0$. In other words, $l_0 = l_1 = \ldots l_{n-1} = 0$,
                as desired.
                % fix an $F_0$-basis $\cl{a}_1,\ldots,\cl{a}_m$ of $L_0$. Since $L/F$ is 
                % unramified, we may lift $\cl{a}_1,\ldots,\cl{a}_m$ to an $F$-basis 
                % $a_1,\ldots,a_m$ of $L$. Now observe that since $a_1\otimes
                % 1,\ldots,a_m\otimes 1$ are an $F(\pi^{1/\infty})$-basis of
                % $L\otimes_FF(\pi^{1/\infty})$ and $f(a_j\otimes 1)=a_j$ for all
                % $j=1,\ldots,m$, in order to prove that $f$ is injective it suffices 
                % to show that  $a_1,\ldots,a_m$ remain linearly independent over 
                % $F(\pi^{1/\infty})$. Suppose the contrary: there exist
                % $\phi_1,\ldots,\phi_m\in L(\pi^{1/\infty})$, not all zero, such that
                % $\sum_{j=1}^ma_j\phi_j=0$. Writing out $\phi_j$ as Laurent series in
                % $\pi^{1/n}$ (for the same $n$) with coefficients $c_{ij} \in L$, we obtain
                % \[\sum_{j=1}^ma_j\left(\sum_{i\geqslant N}c_{ij}\pi^{i/n}\right)=0.\]
                % Here we may assume that each $c_{ij}$ is either zero or a unit. After
                % multiplying both sides by a suitable integer (possibly negative) 
                % power of $\pi^{1/n}$ and reindexing if necessary, we may assume that 
                % $N = 0$ and $\cl{c}_{0j}\neq 0$ for some $j$. 
                % Reduction modulo $\pi$ then yields the identity $\displaystyle 
                % \sum_{j = 1}^m \cl{c}_{0j}\cl{a}_j=0$ in $L_0$, where 
                % the $\cl{c}_{0j}\in F_0$ are the reductions of the $c_{0j}$. 
                % Since the $\cl{a}_j$ are a basis of $L_0$ over $F_0$ and not all 
                % $\cl{c}_{0j}$ are zero, this gives a contradiction. We conclude that 
                % $a_1,\ldots,a_m$ are linearly independent over $F(\pi^{1/\infty})$ 
                % and 
                This shows that $f$ is injective and hence, an isomorphism.  
                
                As a consequence, we have a group isomorphism $\on{Gal}(L/F)\to \on{Gal}\left(L(\pi^{1/\infty})/F(\pi^{1/\infty})\right)$ such that the composition
                \[\on{Gal}(L/F)\longrightarrow \on{Gal}\left(L(\pi^{1/\infty})/F(\pi^{1/\infty})\right) \xrightarrow{\,\on{sp}'\,} \on{Gal}(L_0/F_0)\]
                is $\on{sp}_{L/F}$. This implies that $\gamma$ belongs to the image of $\on{sp}'$. We conclude that $\on{sp}'$ is surjective, as claimed.

                Let $Q$ be the kernel of $\on{sp}'$. Then $Q$ is the intersection of $\on{Gal}(F_{\on{nr}})$ with $\on{Gal}(F(\pi^{1/\infty}))=\on{Gal}(F(\pi^{1/\infty})')$. In other words, $Q$ is the absolute Galois group of the composite field $F_{\on{comp}}$ of $F_{\on{nr}}$ and $F(\pi^{1/\infty})'$. The field $F_{\on{comp}}$ is obtained from $F_{\on{nr}}$ by adjoining all the $\pi^{1/n}$ for $n$ not divisible by $p$, and so $F_{\on{comp}} = F_{\on{mod}}$. Therefore $Q=\on{Gal}(F_{\on{mod}})=P$. This proves the exactness of the top horizontal row, and thus completes the construction of the diagram. 
                
                The existence of a section of $\on{sp}'$ follows from the fact that $P$ is a pro-$p$-group and that $\on{Gal}(F_0)$ has $p$-cohomological dimension equal to $1$; see \cite[\S~II.4.3, Exercise~2(c)]{serre1997galois}. 
        \end{proof}

        \begin{lemma}\label{grothendieck-serre}
                Let $R$ be a complete discrete valuation ring, $F$ and $F_0$ be the fraction field and residue field of $R$, respectively, and $\pi\in R$ be a uniformizer. Fix an algebraic closure $\cl{F}$ of $F$, choose a system $\set{\pi^{1/n}}_{n\geqslant 1}$ of roots of $\pi$ inside $\cl{F}$ such that $(\pi^{1/mn})^m=\pi^{1/n}$ for all $m,n\geqslant 1$, and let $F(\pi^{1/\infty}):= \cup_{n\geqslant 1}F(\pi^{1/n})$. Let $G$ be a smooth affine group scheme over $R$ such that  $G/G^{\circ}$ is $R$-finite. In parts (b) - (d), assume that $G^{\circ}$ is reductive. Then

                \smallskip
                (a) The map $H^1(R,G)\to H^1(F_0,G)$ is bijective.
                
                \smallskip
                (b) The map $H^1(R,G) \to H^1(F,G)$ is injective. 
                
                \smallskip
                (c) The map $H^1(R,G)\to H^1(F(\pi^{1/\infty}),G)$ is injective.
                
                \smallskip
                (d) %Moreover if $G$ is in good characteristic (see Definition~\ref{assume}), 
                Assume further that at least one of the following conditions holds.
                \begin{enumerate}[label=(\roman*)]
                        \item $\on{char}(F_0)=0$, $R=F_0[[\pi]]$ and $G$ is defined over $F_0$.
                        \item There exists a finite abstract group $S$ of order invertible in $F_0$ and an $R$-subgroup embedding $S_R\hookrightarrow G$ such that the induced map 
                        \[H^1(F(\pi^{1/\infty}),S)\to H^1(F(\pi^{1/\infty}),G)\]
                        is surjective.
                        %\[H^1(F_0,S)\to H^1(F_0,G_{F_0}),\qquad H^1(F,S)\to H^1(F,G)\] are surjective.
                \end{enumerate} 
                Then the map $H^1(R,G)\to H^1(F(\pi^{1/\infty}),G)$ is bijective. 
        \end{lemma}

        %In characteristic zero, \Cref{grothendieck-serre} can be readily deduced from work of Florence \cite{florence2006points}; see below. Our contribution is primarily in the positive characteristic case.
        
        %When $G$ is a reductive, this is a special case of the Grothendieck-Serre Conjecture.
        \begin{proof}
                (a) See \cite[Chapter~XXIV, Proposition~8.1]{SGA3I}. Only smoothness of $G$ is needed here.
                
                (b) First we will show that the map $H^1(R, {G})\to H^1(F, {G})$ has trivial kernel. Consider the short exact sequence  of group $R$-schemes 
                \[0\to {G}^{\circ} \to G \to {G}_0\to 0\]
                where ${G}^{\circ}$ is a smooth, connected and reductive groups scheme and ${G}_0$ is a finite group scheme over $R$. 
                Passing to non-Abelian cohomology, we obtain a commutative diagram
                \begin{equation}\label{groth-serre-disc}
                        \begin{tikzcd}
                                {G}_0(R) \arrow[r] \arrow[d,equal] & H^1(R,{G}^{\circ})  \arrow[d] \arrow[r]   & H^1(R, {G}) \arrow[d] \arrow[r]   & H^1(R,{G}_0) \arrow[d]  \\
                                {G}_0(F) \arrow[r] & H^1(F, {G}^{\circ}) \arrow[r]  & H^1(F, {G})  \arrow[r] & H^1(F, {G}_0).
                        \end{tikzcd}
                \end{equation}
                        Since ${G}_0$ is finite over $R$, by the valuative criterion for properness ${G}_0(F)= {G}_0(R)$. If ${T}\to \Spec R$ is a ${G}_0$-torsor over $R$, then $\mathcal{T}$ is also finite over $R$, hence ${T}(F)= {T}(R)$. This shows that the pullback $H^1(R, {G}_0)\to H^1(F,{G}_0)$ has trivial kernel. By the Grothendieck--Serre Conjecture over Henselian discrete valuation rings, due to Nisnevich \cite[Th\'eor\`eme 4.5]{nisnevich1984espaces}, the pullback map 
                        $H^1(R, {G}^{\circ})\to H^1(F, {G}^{\circ})$ also has trivial kernel.     An easy diagram chase now shows that the vertical map $H^1(R, {G})\to H^1(F,{G})$ 
                        has trivial kernel as well.
                
                We will now show that the map $H^1(R, {G})\to H^1(F, {G})$ is actually injective.
                Let $\alpha,\beta \in H^1(R,G)$ be such that $\alpha_F=\beta_F$. We want to show that $\alpha = \beta$. Let $a$ be a cocycle representing $\alpha$, and let 
                $\prescript{a}{}{G}$ be the $R$-group scheme obtained by twisting $G$ by $a$. The fiber of the map $H^1(R, G)\to H^1(F,G)$ over $\alpha_F$ is naturally identified with the kernel of the map $H^1(R,\prescript{a}{}{G}) \to H^1(F,\prescript{a}{}{G})$. Now the same argument as above,
                with $G$ replaced by its twist $\prescript{a}{}{G}$, shows that the map $H^1(R,\prescript{a}{}{G}) \to H^1(F,\prescript{a}{}{G})$ has trivial kernel. Consequently, the fiber of the map $H^1(R, {G})\to H^1(F, {G})$ over $\alpha_F$ is trivial.
        We conclude that $\alpha = \beta$, as desired.

                (c) Let $\alpha,\beta\in H^1(R,G)$ be such that $\alpha_{F(\pi^{1/\infty})}=\beta_{F(\pi^{1/\infty})}$. Our goal is to show that
                $\alpha = \beta$. We have $\alpha_{F(\pi^{1/n})}=\beta_{F(\pi^{1/n})}$ for some $n\geqslant1$. Let $R_n\subset F(\pi^{1/n})$ be the integral closure of $R$ in $F(\pi^{1/n})$. Then $R_n$ is a discrete valuation ring with fraction field $F(\pi^{1/n})$, residue field $F_0$, and uniformizer $\pi^{1/n}\in R_n$. By (b) the map $H^1(R_n,G)\to H^1(F(\pi^{1/n}),G)$ is injective, hence $\alpha_{R_n}=\beta_{R_n}$. We have a commutative diagram
                \[
                \begin{tikzcd}
                        H^1(R,G) \arrow[r] \arrow[d,"\wr"] & H^1(R_n,G) \arrow[d,"\wr"] \\
                        H^1(F_0,G) \arrow[r,equal] & H^1(R_n/(\pi^{1/n}),G),
                \end{tikzcd}
                \]
                where the left and right vertical arrows are induced by reduction modulo $\pi$ and $\pi^{1/n}$, respectively. 
                Both are isomorphisms by part (a). We conclude that $\alpha_{R_n}=\beta_{R_n}$ and consequently, $\alpha=\beta$, as desired.
                
                (d)  In view of part (c), it suffices to show that      the map $H^1(R, G) \to H^1(F(\pi^{1/\infty}), G)$ is surjective. If (i) holds, then the surjectivity of this map follows from \cite[Proposition 5.4]{florence2006points}. 
                
                From now on we will assume that condition (ii) of part (d) holds. We will break up the proof into several steps given by the claims below.
            Recall that the homomorphism  $\on{sp}' \colon \on{Gal}(F(\pi^{1/\infty})) \to \on{Gal}(F_0)$ has a section by~\Cref{big-diag}. Denote a section
            of $\on{sp}'$ by $\sigma:\on{Gal}(F_0)\to \on{Gal}(F(\pi^{1/\infty}))$, and let $\Gamma$ be its image.
                
                \begin{claim}\label{claim1}
                        The pullback map
                        $\sigma^*\colon H^1(F(\pi^{1/\infty}),S)\to H^1(\cl{F}^\Gamma,S)$
                        is injective for every finite discrete group $S$ of order invertible in $F_0$.
                \end{claim}     
                
                \begin{proof}[Proof of \Cref{claim1}] Given that $S$ is a finite discrete group, we may identify
                $H^1(F(\pi^{1/\infty}),S)$ with $\on{Hom}\left(\on{Gal}(F(\pi^{1/\infty})),S\right)/\sim \,$, $H^1\left(\cl{F}^\Gamma,S\right)$
                with $\on{Hom}(\Gamma,S)/\sim \, $, and $\sigma^*$ with the map
                \[\raisebox{.3em}{$\on{Hom}(\on{Gal}\left(F(\pi^{1/\infty})),S\right)$}\left / \raisebox{-.3em}{$\sim$}\right. \longrightarrow  \raisebox{.3em}{$\on{Hom}(\Gamma,S)$}\left/\raisebox{-.3em}{$\sim$} \right. \]
                induced by the restriction $\Gamma=\on{Gal}\left(\cl{F}^\Gamma\right)\to \on{Gal}\left(F(\pi^{1/\infty})\right)$. Here the symbol $\sim$ stands for
                the equivalence relation given by conjugation by an element of $S$.
        
                Let \[a,b \colon \on{Gal}\left(F(\pi^{1/\infty})\right)\longrightarrow S\] be group homomorphisms such that $a|_{\Gamma}\sim b|_{\Gamma}$. 
                After replacing $b$ by an $S$-conjugate, we may assume that $a|_{\Gamma}=b|_{\Gamma}$. Our goal is to prove that $a=b$. Let 
                \[H:=\set{\gamma\in\on{Gal}\left(F(\pi^{1/\infty})\right) \; \big| \; a(\gamma)=b(\gamma)}\] be the equalizer of $a$ and $b$. Since $a|_{\Gamma}=b|_{\Gamma}$, we know that $\Gamma\subset H$. By \Cref{big-diag}, $\on{Gal}(F(\pi^{1/\infty}))$ is a semi-direct product of $P$ and $\Gamma$, where $P$ is a 
                pro-$p$ group. Hence, $\on{Gal}(F(\pi^{1/\infty}))=P\cdot H$ and consequently,
                \begin{equation}\label{h1}[\on{Gal}(F(\pi^{1/\infty})):H]=[P \, :  \, (H\cap P)]\text{ is a power of $p$.}\end{equation}
                On the other hand, consider the group homomorphism \[\phi:\on{Gal}(F(\pi^{1/\infty}))\to S\times S\ \text{given by}\ \phi(\gamma)=(a(\gamma),b(\gamma)),\] and let \[H_{a,b}:=\on{ker}\phi=\on{ker}a\cap \on{ker}b.\] Then $\phi$ factors through an injective homomorphism of finite groups \[\on{Gal}\left(F(\pi^{1/\infty})\right)/H_{a,b}\lhook\joinrel\longrightarrow S\times S.\]
                Since $S$ has order prime to $p$, the index $[\on{Gal}(F(\pi^{1/\infty}):H_{a,b}]$ is not divisible by $p$. Since $H_{a,b}\subset H$, we deduce that
                \begin{equation}\label{h2}
                        [\on{Gal}(F(\pi^{1/\infty})):H] \text{ is not divisible by $p$}.\end{equation} 
                It follows from (\ref{h1}) and (\ref{h2}) that $H=\on{Gal}(F(\pi^{1/\infty}))$, that is, $a=b$. 
                \end{proof}
                
                \begin{claim}\label{claim2}
                        For every finite discrete group $S$, the pullback map  $H^1(R,S)\to H^1\left(\cl{F}^{\Gamma},S\right)$ is bijective. 
                \end{claim}
                
                \begin{proof}[Proof of~\Cref{claim2}] Consider the composition \[\phi:\Gamma\lhook\joinrel\longrightarrow \on{Gal}(F)\longrightarrow \pi_1(\Spec R,\cl{\eta}),\] where $\cl{\eta}$ is a geometric point lying above the generic point of $\Spec R$. The pullback map $H^1(R,S)\to H^1(\cl{F}^{\Gamma},S)$ may be identified with the map
\[\raisebox{.3em}{$\on{Hom}(\pi_1(\Spec R,\cl{\eta}),S)$}\left/\raisebox{-.3em}{$\sim$}\right.  \longrightarrow  \raisebox{.3em}{$\on{Hom}(\Gamma,S)$}\left/\raisebox{-.3em}{$\sim$}\right.\] induced by precomposition with $\phi$. Here once again, $\sim$ denotes the equivalence relation given by conjugation by an element of $S$. 
                By construction, the composition \[\Gamma\xrightarrow{\ \phi\ } \pi_1(\Spec R,\cl{\eta})\xrightarrow{\ \sim\ } \on{Gal}(F_0)\] is the restriction of $\on{sp}$ to $\Gamma$, and so it is an isomorphism, with inverse $\sigma$. We conclude that $\phi$ is also an isomorphism. 
                % This completes the proof of \Cref{claim2}.
                \end{proof}
                
                \begin{claim}\label{claim3}
                        The map $H^1(R, S) \to H^1\left(F(\pi^{1/\infty}),S\right)$ is bijective for every finite discrete group $S$ of order invertible in $F_0$.
                \end{claim}
                
                \begin{proof}[Proof of \Cref{claim3}] Injectivity follows from (c), and surjectivity from the commutativity of the triangle
                \[
                \begin{tikzcd}
                        H^1(R,S) \arrow[r] \arrow[dr,"\rotatebox{-25}{$\sim$}"']  & H^1\left(F(\pi^{1/\infty}),S\right) \arrow[d,hook] \\
                        & H^1\left(\cl{F}^{\Gamma},S\right).       
                \end{tikzcd}
                \]
                Here the vertical map is injective by \Cref{claim1} and the diagonal map is bijective by \Cref{claim2} applied to~$S$. % This proves \Cref{claim3}.
                \end{proof}
                                
                We are now ready to complete the proof of \Cref{grothendieck-serre}(d) under assumption (ii). Injectivity follows from~(c). Let $S$ is the finite group provided by assumption (ii), such that the natural map $H^1(F(\pi^{1/\infty}),S)\to H^1(F(\pi^{1/\infty}),G)$ is surjective. By \Cref{claim3}
                the map $H^1(R,S)\to H^1(F(\pi^{1/\infty}),S)$ is also surjective, and part (d) follows.
        \end{proof}

        \section{Proof of Theorem \ref{puiseux} (i) and (ii)}
        \label{sect.puiseux(i)-(ii)}
        %In this section we prove \Cref{puiseux} under assumptions (i) and (ii). 
        
        \begin{lemma}\label{e=1}
                Let $A\subset R$ be an inclusion of discrete valuation rings. Let $\mathfrak{m}_A$ and $\mathfrak{m}_R$ be the maximal ideals of $A$ and $R$, respectively, and assume that $\mathfrak{m}_A\subset \mathfrak{m}_R$. Then there exists a discrete valuation ring $A\subset B\subset R$ with maximal ideal $\mathfrak{m}_B$ such that $\mathfrak{m}_BR=\mathfrak{m}_R$ and such that the induced map $A/\mathfrak{m}_A\to B/\mathfrak{m}_B$ is an isomorphism.
        \end{lemma}
        
        \begin{proof}
                Let $\pi\in \mathfrak{m}_R-\mathfrak{m}_R^2$ be a uniformizer, and set $B:=A[\pi]_{(\pi)}\subset R$. Since $A$ is noetherian and $B$ is a localization of a finitely generated $A$-algebra, $B$ is also noetherian. Thus $B$ is a local noetherian domain whose maximal ideal is principal, and consequently, a discrete valuation ring. 
                
                Since $\pi$ is a generator of $\mathfrak{m}_R$, it is clear that $\mathfrak{m}_BR=\mathfrak{m}_R$. The map $A/\mathfrak{m}_A\to B/\mathfrak{m}_B$ is surjective because $\pi$ goes to zero in $B/\mathfrak{m}_B$, and is injective because $A/\mathfrak{m}_A$ is a field.
                %It is also injective: if $a\in A$ belongs to $\mathfrak{m}_B$, then it belongs to $\mathfrak{m}_R$, but by \cite[Lemma 07BJ]{stacks-project} we have $\mathfrak{m}_R\cap A=\mathfrak{m}_A$, hence $a\in \mathfrak{m}_A$.
        \end{proof}

        \begin{proof}[Proof of \Cref{puiseux} assuming (i) or (ii)] 
                Choose a discrete valuation ring  $B$ between $A$ and $R$ as in~\Cref{e=1}, so that $\mathfrak{m}_BR=\mathfrak{m}_R$ and the induced map 
                $k_0 = A/\mathfrak{m}_A\to B/\mathfrak{m}_B$ is an isomorphism.
                Let $\hat{B}$ be the completion of $B$. Since $R$ is complete, 
                the universal property of the completion gives rise to an embedding
                $\hat{B} \hookrightarrow R$. Denote the fraction field of $\hat{B}$ by $l$.
                By \Cref{lem.prel1}(b), $\ed_{k}(\alpha_K) \geqslant \ed_l(\alpha_K)$. Thus in order to
                prove the inequality 
                \begin{equation} \label{e.puiseux}
                \ed_{k_0}(\alpha_{K_0}) \leqslant \ed_k(\alpha_K), 
                \end{equation}
                of~\Cref{puiseux}, it suffices to show that $\ed_{l}(\alpha_{K_0}) \leqslant \ed_l(\alpha_K)$.
                Thus for the purpose of proving Inequality~\eqref{e.puiseux}, we may replace
                $A$ by $\hat{B}$. In other words, we may assume
                we may assume that $\mathfrak{m}R=\mathfrak{m}_R$ is the maximal ideal of $R$.
            
            Once again, let $\pi\in \mathfrak{m}_R-\mathfrak{m}_R^2$ be a uniformizing parameter.
                By definition of $\ed_k(\alpha_K)$, there exists an intermediate field $k\subset F\subset K$ such that
                \begin{equation}\label{e.ed-alphat}
                        \ed_{k}(\alpha_K)=\trdeg_k(F)
                \end{equation}
                and $\alpha$ descends to $F$. Let $v \colon K^{\times}\to \Z$ be the valuation determined by $\pi$. Since $F$ contains $k$, $v|_{F^{\times}}$ is surjective. Let $O\subset F$ be the valuation ring of $v|_{F^{\times}}$, and $F_0$ be the residue field of the restriction of $v$ to $F$. 
                The inclusion $A\subset O$ induces an inclusion $k_0\subset F_0$. Moreover, by \cite[Lemma 7.1]{reichstein2020essential} we have
                \begin{equation} \label{e.residue-field} \trdeg_{k}F\geqslant \trdeg_{k_0}F_0.
                \end{equation}
                Let $\hat{O}$ be the completion of $O$, and let $\hat{F}$ be the fraction field of $\hat{O}$, that is, the completion of $F$ as a valued field. Since $K$ is complete, the universal property of the completion gives a unique field embedding $\hat{F}\hookrightarrow K$ extending the inclusion $F\subset K$. We view $\hat{F}$ as a subfield of $K$ via this embedding, so that $F\subset \hat{F}\subset K$, $O\subset \hat{O}\subset R$, and the residue field of $\hat{F}$ is $F_0$. 
                
                Since $\pi$ is a uniformizer in $R$, it is also a uniformizer in $O$ and in $\hat{O}$. 
                Let $\cl{K}$ be an algebraic closure of $K$, and fix a system $\set{\pi^{1/m}}_{m\geqslant1}$ of roots of $\pi$ in $\cl{K}$ such that $(\pi^{1/mn})^n=\pi^{1/m}$ for all $m,n\geqslant1$. Such a system of roots exists by \Cref{system}(a).  We then define the following subfields of $\cl{K}$:
                \[K(\pi^{1/\infty}):=\bigcup_{n\geqslant1} K(\pi^{1/n})\quad\text{and}\quad F(\pi^{1/\infty}):=\bigcup_{n\geqslant1} F(\pi^{1/n}).\]
                Thus we have a commutative diagram
                \[
                \begin{tikzcd}
                        F_0 \arrow[d,hook] & \arrow[l]  \hat{O} \arrow[d,hook]\arrow[r,hook] &   \arrow[r,hook] \hat{F} \arrow[d,hook] & \hat{F}(\pi^{1/\infty}) \arrow[d, hook]  \\
                        K_0 & \arrow[l] R \arrow[r,hook] & K \arrow[r,hook] & K(\pi^{1/\infty}).
                \end{tikzcd}
                \]
                Passing to Galois cohomology, we obtain the following commutative diagram
                \begin{equation*}\label{key-diag}
                        \begin{tikzcd}
                                H^1(F_0,G) \arrow[d] & \arrow[l,swap, "\sim"]  H^1(\hat{O},G) \arrow[d]\arrow[r,"\sim"] & H^1\left(\hat{F}(\pi^{1/\infty}),G\right) \arrow[d]  \\
                                H^1(K_0,G) & \arrow[l,swap,"\sim"] H^1(R,G) \arrow[r,"\sim"] & H^1\left(K(\pi^{1/\infty}),G\right)
                        \end{tikzcd}
                \end{equation*}
                where the horizontal maps on the left are isomorphisms by \Cref{grothendieck-serre}(a), and those on the right are isomorphisms by \Cref{grothendieck-serre}(d). (This is where assumptions (i) and (ii) are used.) 
                We deduce that $\alpha_{K_0}$ descends to $F_0$, and thus
                \[\trdeg_{k_0}F_0\geqslant \ed_{k_0}\alpha_{K_0}.\]
                By \eqref{e.ed-alphat} and \eqref{e.residue-field}, we conclude that 
                \[\ed_{k}\alpha_K = \trdeg_{k}F\geqslant \trdeg_{k_0}F_0\geqslant \ed_{k_0}\alpha_{K_0}.\]
                This completes the proof of the Inequality~\eqref{e.puiseux}.
                
                In the case, where $A=k_0[[t]]$ and $G_{A}$ is defined over $k_0$, \Cref{lem.prel1}(a) implies that the reverse inequality is also true
                and thus $\ed_{k}\alpha_K=\ed_{k_0}\alpha_{K_0}$. This completes the proof of \Cref{puiseux} under assumptions (i) and (ii).
                (Note that at the beginning of the proof of~\eqref{e.puiseux} we replaced $A$ by
                $\hat{B}$. For the reverse inequality we work with $A= k_0[[t]]$ directly, here we do not need $\hat{B}$.)
        \end{proof}
       
        We conclude this section with the following remarks on conditions (i) and (ii) of \Cref{puiseux}. 
        
\begin{rmks} \label{rem(i)-(ii)}
\leavevmode
\begin{enumerate}
\item By Cohen's Structure Theorem, when $\on{char}(k_0)=0$ there exists an isomorphism $A\simeq k_0[[t]]$ inducing the identity on residue fields. Such an isomorphism is not unique. One may restate (i) as follows: we may identify $A$ with $k_0[[t]]$ in such a way that $G$ is defined over the subfield of constants $k_0\subset k_0[[t]]$.
\item In (i), the fact that $G$ is constant is not automatic in general; see \cite[Chapitre~XIX, \S~5]{SGA3III} for an example, where $G_{k((t))}$ is an \'etale form of $\on{PGL}_2$ but the fiber $G_0$ at $t=0$ is solvable with two connected components.
                        
                        However, if $\on{char}(k_0)=0$ and $G$ is reductive over $A=k_0[[t]]$, then $G$ is constant. To see this, note that by \cite[Chapitre~XXIV, Corollaire~1.18]{SGA3III}, letting $G_{\on{split}}$ denote the split form of $G$ over $A$, there exists an \'etale $\on{Aut}(G_{\on{split}})$-torsor $P$ such that twisting $G_{\on{split}}$ by $P$ yields $G$. By \cite[Chapitre~XXIV, Th\'eor\`eme~1.3]{SGA3III}, $\on{Aut}(G_{\on{split}})$ is smooth over $A$, hence by \cite[Chapitre~XXIV, Proposition~8.1]{SGA3I} the torsor $P$ is defined over $k_0$. Since $G_{\on{split}}$ is split, it is defined over $\Z$, hence over $k_0$. Thus $G$, being the twist of a group defined over $k_0$ by a torsor defined over $k_0$, is also defined over $k_0$.
\item In (ii), if we have an $A$-group embedding $S_{A}\subset G_{A}$, then we also have a $k_0$-group embedding $S_{k_0}\subset G_{k_0}$. However, the converse does not necessarily hold: the natural group homomorphism $G(A)\to G(k_0)$ is surjective (by Hensel's lemma, since $G$ is $A$-smooth) but not necessarily injective.
\end{enumerate}                          
\end{rmks}

        \section{Proof of Theorem \ref{thm.split-red}}
        
        Recall that a commutative ring with identity is said to be semilocal if it has finitely many maximal ideals.
        
        \begin{prop}\label{red-structure-split}
                Let $G$ be a split reductive group over $\Spec\Z$ of rank $r \geqslant 0$. Let $n$ be the order of the Weyl group of $G$. Then there exists a finite flat $\Z$-subgroup $S\subset G$ % of degree $2^r n$ 
                with the following properties.
\begin{enumerate}[label=(\alph*)]
\item For every semilocal ring $R$ the natural map $H^1(R,S)\to H^1(R,G)$ is surjective.
\item If a prime number $p$ divides the degree of $S$, then $p$ divides $2^r n$.
\item $S$ is constant over $R = \Z[1/2^r, \zeta]$ and $\zeta$ is a primitive root of unity of degree $2^rn$. In other words, $S_R$ is the constant $R$-group scheme associated to an abstract finite group. 
\end{enumerate}
\end{prop}
        
        \begin{proof}
                Let $T\subset G$ be a split maximal torus over $\Z$, let $N\subset G$ be the normalizer of $T$, and let $W:=N/T$ be the Weyl group scheme. Since $W$ is \'etale over $\Z$, it is a disjoint union of copies of $\Spec \Z$. Since $\on{Pic}(\Z)=0$ and $T$ is split, we have $H^1(\Z,T)=0$. It follows that we have a short exact sequence
                \[1\to T(\Z)\to N(\Z)\to W(\Z)\to 1.\]
                Since $T(\Z)= (\G_{\on{m}}(\Z))^r = \set{\pm 1}^r$, the finite group $N(\Z)$ has order $2^rn$. Let $\Gamma\subset G$ be the scheme-theoretic closure of $N(\Z)\subset N(\Q)$ inside $G$. 
                
                \begin{claim}\label{claim5.1.new}
                The $\Z$-subgroup $\Gamma$ is reduced and finite flat  of degree $2^rn$ over $\Z$.
                \end{claim}
                
                \begin{proof}[Proof of \Cref{claim5.1.new}] Since $\Gamma$ is defined as the scheme-theoretic closure of the constant subscheme $N(\Z)_{\Q}\subset G_{\Q}$, by \cite[\'Equation~1.2.6(2)]{bruhat1984groupes} we have $\Gamma_{\Q}=N(\Z)_{\Q}$. Hence $\Gamma_{\Q}$ is finite over $\Q$ 
                % is 
                of degree $2^rn$.
         Note that $T\cap \Gamma\subset T$ is finite over $\Z$; otherwise $\Gamma$ would contain a $\Z$-subgroup isomorphic to $\G_{\on{m}}$ and hence $\Gamma(\Q)$ would be infinite, a contradiction. The quotient $T/(T\cap \Gamma)$ is also finite over $\Z$ because it is a closed subgroup of $W$. Since $\Gamma$ is an extension of $T/(T\cap \Gamma)$ by $T\cap \Gamma$, it is finite over $\Z$ as well. In particular, the degree of $\Gamma$ over $\Z$ is equal to the degree of 
         $\Gamma_{\Q}$ over $\Q$, which is $2^rn$. Since $N(\Z)_{\Q}$ is reduced, $\Gamma$ is also reduced, and hence $\Gamma$ is flat over $\Z$; 
         see~\cite[Lemma 4.3.9]{liu2002algebraic}. % This proves \Cref{claim5.1.new}.\hfill\qed
         \end{proof}
                
        %\smallskip
         \begin{claim}\label{claim5.1.1}
            Let $s \colon \Spec \Z\to N$ be a $\Z$-point of $N$. Then $s$ factors through $\Gamma$.  
         \end{claim}    
         %\smallskip

         \begin{proof}[Proof of \Cref{claim5.1.1}] 
         Clearly $s_{\Q}:\Spec \, \Q\to N_{\Q}$ factors through $\Gamma_{\Q}$. Moreover, since $\Gamma$ is finite over $\Z$, by the existence part of the valuative criterion for properness $s_{\Z_{(p)}}:\Spec\, \Z_{(p)}\to N_{\Z_{(p)}}$ factors through $\Gamma_{\Z_{(p)}}$ for every prime $p$. By the uniqueness part of the valuative criterion for properness, this factorization is unique for each $p$. Therefore, these local factorizations glue to  a global factorization of $s$ through $\Gamma$ over $\Z$. % This proves \Cref{claim5.1.1}.\hfill\qed
         \end{proof}
                
        %\smallskip
        \begin{claim}\label{claim5.1.2} 
        %Let $\Gamma' \subset \Gamma$ be the union of the $2^rn$ irreducible components given by sections of $\Gamma\to\Spec \, \Z$. Then $\Gamma_{\Z[1/2^r]}=\Gamma'_{\Z[1/2^r]}$.
        The group scheme $\Gamma$ is the union of the $2^rn$ irreducible components given by sections of $N\to\Spec \, \Z$. Moreover, the restriction of $\Gamma$ to $\Spec \,\Z[1/2^r]$ is constant.
        \end{claim}

        \begin{proof}[Proof of \Cref{claim5.1.2}] 
        Let $\Gamma'$ be the closed subscheme of $N$ given by the union of the images of the $2^rn$ sections of $N\to\Spec \, \Z$. (Note that $N$ is separated, hence every section of $N\to \Spec(\Z)$ is a closed immersion. Therefore $\Gamma'$ is closed.) We have $\Gamma'\subset \Gamma$ by \Cref{claim5.1.1}. By \Cref{claim5.1.new} there exist a $\Z$-algebra $R$, finite and free of rank $2^rn$, and an ideal $I\subset R$ such that $\Gamma=\Spec(R)$ and $\Gamma'=\Spec(R/I)$. Since $\Gamma'$ is the union of $2^rn$ copies of $\Spec(\Z)$ over $\Z$, the rank of $R/I$ is also $2^rn$. It follows that $I$ is torsion. Since $R$ is free, this implies that $I=0$, hence $\Gamma'=\Gamma$. 
        
        In order to prove the second part of \Cref{claim5.1.2}, we consider the commutative diagram of short exact sequences
                \[
                \begin{tikzcd}
                        0 \arrow[r] & T(\Z) \arrow[r] \arrow[d, hook] & N(\Z) \arrow[r] \arrow[d] & W(\Z) \arrow[r] \arrow[d,"\wr"] & 0 \\
                        0 \arrow[r] & T(\F_p) \arrow[r] & N(\F_p) \arrow[r]  & W(\F_p) \arrow[r] & 0
                \end{tikzcd}
                \]
                where $\F_p$ is the field of $p$ elements and $p$ is prime to $2^r$. The vertical map on the left is the inclusion $\set{\pm 1}^r\hookrightarrow (\F_p^{\times})^r$, and in particular is injective. On the other hand, the map $W(\Z)\to W(\F_p)$ is an isomorphism because $W$ is a disjoint union of copies of $\Spec \, \Z$. It follows that $N(\Z)\to N(\F_p)$ is injective. Consequently, $\Gamma_{\F_p}=\Gamma'_{\F_p}$ is a disjoint union of $2^rn$ copies of $\Spec \F_p$, which means that the irreducible components of $\Gamma$ do not intersect above $p$. Therefore the restrictions to $\Spec\,\Z[1/2^r]$ of the $2^rn$ irreducible components of $\Gamma$ are pairwise disjoint. This concludes the proof of~\Cref{claim5.1.2}.
                \end{proof}
        
                \smallskip
                For each positive integer $m\geqslant 1$, let $\phi_m \colon T\to T$ be the $m^\mathrm{th}$ power map given by $t\mapsto t^m$.  As usual, we will denote the kernel of $\phi_m$ by $T[m]$. For each $m \geqslant 1$, $T[m]$ is a flat finite normal subgroup of $N$ of degree $m^r$.
                
                %\smallskip
                \begin{claim}\label{claim5.1.3} We have $T\cap \Gamma =T[2]$. Moreover, $T[2]$ becomes constant over $\Spec\,\Z[1/2^r]$.
                \end{claim}
                
        \begin{proof}[Proof of \Cref{claim5.1.3}] 
        We have $T(\Z) = T[2](\Z)=\set{\pm 1}^r$ and $T[2]\simeq (\Spec\,\Z[t]/(t^2-1))^r$. It follows that the union of the $2^r$ sections of $T\to \Spec\,\Z$ is equal to $T[2]$ and that $T[2]$ becomes constant over $\Spec\,\Z[1/2^r]$.
        
        By \Cref{claim5.1.2}, $\Gamma\to\Spec\,\Z$ is the union of all sections of $N\to \Spec\,\Z$. Let $s:\Spec\,\Z\to N$ be a section. The composition of $s$ with the projection $N\to W$ is a section $w:\Spec\,\Z\to W$. If $w$ is trivial, then the image of $s$ is contained in $T$, and if $w$ is non-trivial, then the image of $s$ is disjoint from $T$. This shows that $T\cap\Gamma$ is the union of the sections of $N \to \Spec\,\Z$ which map to the identity section of $W$, {\it i.e.} the sections of $T \to \Spec\,\Z$. Therefore $T\cap\Gamma=T[2]$. This proves \Cref{claim5.1.3}.
        \end{proof}
        
\smallskip
We are now ready to finish the proof of~\Cref{red-structure-split}.
Let $S$ be the subgroup of $N$ generated by $\Gamma$ and $\phi_n^{-1}(\Gamma\cap T)$. By construction, $S$ is finite flat over $\Z$. 
\begin{enumerate}[label=(\alph*)]
\item The map $H^1(R,S)\to H^1(R,G)$ is surjective for every semilocal ring $R$; see~\cite[Proposition~3.1]{chernousov2008reduction}.
\item The subgroup of $N$ generated by $T[m]$ and $\Gamma$ is finite and flat. Its degree divides $\on{deg}(T[m])\cdot \on{deg}(\Gamma)=m^r \cdot 2^r n$. Over $\Spec \Z[1/2^r]$, $\Gamma\cap T=T[2]$ by \Cref{claim5.1.3}. Hence over $\Spec \Z[1/2^r]$, $\phi_n^{-1}(\Gamma\cap T) = \phi_n^{-1}(T[2]) = T[2n]$. Setting $m = 2 n$, we see that the degree of $S$ divides $|T[2^r n]| \cdot |\Gamma| = (2 n)^r \cdot (2^r n)$, and part (b) follows.
\item  Over $R = \Z[1/2^r, \zeta]$, both $T[2^r n]$ and $\Gamma$ become constant and hence, so does $S$.
\end{enumerate}
This concludes the proof of~\Cref{red-structure-split}.
\end{proof}

      \begin{cor} \label{cor.versal} Let $G$ be a split reductive group over $\Spec(\Z)$. Then there exists a $G$-torsor
      $\tau \colon X \to B$, where $B$ is a smooth irreducible $\Z$-scheme of finite type and for every field $k$ we have 
      $\ed_k(X_{k(B_k)}) = \ed_k(G_k)$. Here $X_{k(B_k)}$ is obtained by restricting the $G_k$-torsor $X_k \to B_k$ 
      to the generic point of $B_k$.
      \end{cor}
      
      \begin{proof} By \cite[Expos\'e~$\textrm{VI}_\textrm{B}$, Proposition~13.2]{SGA3I}
      there exists an embedding $G \hookrightarrow \GL_{n,\Z}$ over $\Z$. Since $\on{GL}_{n,\Z}$ and $G$ are reductive, by \cite[Theorem 9.4.1]{alper2014adequate} the sheaf-theoretic quotient $\GL_{n,\Z}/G$ is represented by an affine $\Z$-scheme. It is clear that $\on{GL}_{n,\Z}/G$ is of finite type over $\Z$ and irreducible. It is flat over $\Z$ by \cite[Tag 02JZ]{stacks-project} and it has smooth fibers, hence it is smooth over $\Z$. We can now take $X = \GL_{n,\Z}$, $B = \GL_{n,\Z}/G$ and $\tau$ the quotient map. 
      For any field $k$, the $G_k$-torsor $X_k \to B_k$ will then be $\GL_{n, k} \to \GL_{n, k}/G_k$. Here $G_k$
      acts on $X_k = \GL_{n, k}$ by right translations. Since this action is birationally isomorphic to the linear action of $G_k$
      on the affine space of $n \times n$ matrices over $k$, we conclude that $\ed_k(X_{k(B_k)}) = \ed_k(G_k)$; see \Cref{lem.prel2}(a).
      \end{proof}
        
        \begin{proof}[Proof of Theorem \ref{thm.split-red}] 
        By~\Cref{lem.prel2}(c),
         $\ed_k(G) \geqslant \ed_{\overline{k}}(G)$ where $\overline{k}$ is the algebraic closure of $k$. Hence, we may, without loss 
        of generality, replace $k$ by $\overline{k}$ and thus assume that $k$ is algebraically closed. 
        By \Cref{lem.prel2}(d), $\ed_k(G) = \ed_{\overline{\mathbb Q}}(G)$
        for any algebraically closed field $k$ of characteristic $0$, where $\cl{\mathbb Q}$ 
        is the field of algebraic numbers.
        Thus we may assume without loss of generality that 
        \begin{equation} \label{e.Qbar}
        k = \overline{\mathbb Q}.
        \end{equation}
        By \Cref{lem.prel2}(e) there exists a number field $L \subset \overline{\mathbb{Q}}$ 
        such that 
    \begin{equation} \label{e.E}
    \ed_{L}(G)= \ed_{\overline{\mathbb Q}}(G).
    \end{equation}
    We may assume that $L$ contains $\zeta$, where $\zeta$ is a primitive root of unity of degree $2^r n$. Otherwise we can 
    simply replace $L$ by $L(\zeta)$. Indeed, by \Cref{lem.prel2},
    \[ \ed_{L}(G) \geqslant \ed_{L(\zeta)}(G) \geqslant \ed_{\overline{\Q}}(G), \]
    and \eqref{e.E} forces $\ed_{L(\zeta)}(G) = \ed_{\overline{\Q}}(G)$.
    
    Let $\mathcal{O}_L$ be the ring of algebraic integers in $L$ and $\frp \subset \mathcal{O}_L$ be a prime ideal lying over $p$. 
    Note that $\zeta$, being an algebraic integer, lies in $\mathcal{O}_L$. 
    Let $A$ be the completion of the localization of $\mathcal{O}_L$ at $\frp$. 
    Then $A$ is a complete local ring whose fraction field $F$ contains $L$ 
       and whose residue field $f$ is a finite field of characteristic $p$.  Recall that  
       $\zeta \in \mathcal{O}_L \subset A$.
       
       Now let $X \to B$ be the $G$-torsor constructed in \Cref{cor.versal}.  Consider the Cartesian diagram
        \[ \xymatrix{   & X_A \ar@{->}[d]  \ar@{->}[r] & X \ar@{->}[d]           \\          
                 \Spec(F) \ar@{->}[rd]^{\alpha}             & B_A \ar@{->}[r] \ar@{->}[d] &   B  \ar@{->}[d]   \\
%                        \Spec(K) \ar@{->}[rd] &                 &                 \\
                          \Spec(f) \ar@{->}[r]           & \Spec(A)    \ar@{->}[r] &  \Spec(\Z)   }\]
        where $\Spec(F)$ is the generic point of $\Spec(A)$ and $\Spec(f)$ is the closed point.
By \Cref{puiseux}(ii), we have
\begin{equation} \label{e.mixed1}
\ed_F(X_{F}) \geqslant \ed_f(X_{f}).
\end{equation}
Note that condition (ii) of \Cref{puiseux} is satisfied here: a finite subgroup $S \subset G_A$ such that the morphism
$H^1(L, S) \to H^1(L, G)$ is surjective for every field $L$ containing $F$  
is constructed in~\Cref{red-structure-split}. Our assumption that $p$ does not divide $2^r n$ allows us to factor the inclusion $\Z \hookrightarrow A$ 
through $\Z[1/2^r, \zeta]$. Thus $S_A$ is constant by~\Cref{red-structure-split}(c), as required.

By \Cref{cor.versal}, we infer
\begin{equation} \label{e.mixed2}
\ed_F(X_{F}) = \ed_F(G_F) \quad \text{and} \quad \ed_{f}(X_{f}) = \ed_{f}(G_{f}).
\end{equation}
We finally get
\[ \ed_k(G_k) \stackrel{(a)}{=} \ed_{\overline{\Q}}(G_{\overline{\Q}}) \stackrel{(b)}{=}
\ed_L(G_L) \stackrel{(c)}{\geqslant} \ed_F(G_F) \stackrel{(d)}{\geqslant} \ed_{f}(G_{f}) \stackrel{(e)}{\geqslant} \ed_{k_0}(G_{k_0})  \]
as desired. Here (a) and (b) follow from our assumptions \eqref{e.Qbar} and \eqref{e.E} on $k$ and $L$, (c) from \Cref{lem.prel2}(c) using
the inclusion $L \subset F$, 
(d) from~\eqref{e.mixed1} and \eqref{e.mixed2}, and (e) again from \Cref{lem.prel2}(c) using the inclusions $f \hookrightarrow \overline{\F_p} \hookrightarrow k_0$.
This completes the proof of \Cref{thm.split-red}.
        \end{proof}

        \section{Proof of Theorem \ref{puiseux}(iii)}
        \label{sect.puiseux(iii)}
        
        We have not been able to adapt the Galois-theoretic approach of \Cref{sect.puiseux(i)-(ii)} to prove \Cref{puiseux} in case (iii). 
        The proof below is based on valuation-theoretic methods in the spirit of~\cite{brosnan2018essential}; see 
        also~\cite[Lemma~6.2]{fakhruddin2021finite}.

        \begin{prop} \label{prop1b} Let $L$ be a field and $v \colon L^{\times} \to \mathbb Z$ be a discrete valuation. 
        Assume that a finite group $G$ acts faithfully on $L$ and that $v$ is invariant under this action.
        Assume further that     $k$ is a subfield of $K = L^G$ such that $v(k^{\times}) = \Z$. 
        Denote the residue fields of $L$, $K$, and $k$ by $L_0$, $K_0$ and $k_0$, respectively.
        Finally, assume that $G$ is weakly tame at $p = \Char(k_0) \geqslant 0$. Then
\begin{enumerate}[label=(\alph*)]
\item $G$ acts faithfully on $L_0$, and
\item $\ed_{k}(L/K) \geqslant \ed_{k_0}(L_0/K_0)$.
\end{enumerate}        
        \end{prop}      
        
        \begin{proof} (a) Let $\Delta$ be the kernel of the $G$-action on $L_0$. By our assumption, $v(k^{\times}) = \Z$
                and thus $v(K^{\times}) = \Z$. If $\Char(k) = 0$, then \cite[Proposition 2.3]{brosnan2018essential}
                tells us that $\Delta = 1$ and we are done. If $\Char(k) = p > 0$, then \cite[Proposition 2.3]{brosnan2018essential}
                tells us that $\Delta$ is a $p$-group. Since we are assuming that $G$ is weakly tame at $p$, this implies that
                $\Delta = 1$ in this case as well. 
        
                (b) Suppose $\ed_k(L/K) = d$. This means that there exists an intermediate field $k \subset K' \subset K$ and a
                $G$-Galois extension $L'/K'$ such that $L = L' \otimes_{K'} K$ and $\trdeg_k(K') = d$. 
                Let $K'_0$ and $L'_0$ be the residue fields for the restriction of $v$ to $K'$ and $L'$, respectively.
                
                By part (a), the $G$-action on $L'_0$ is faithful. We claim that $(L_0')^G = K_0'$.
                Clearly $K'_0 \subset (L'_0)^G \subset L'_0$ and, since $G$ acts faithfully on $L'_0$, we have $[L'_0 : (L'_0)^G] = |G|$. 
                On the other hand, by \cite[Corollary XII.6.3]{lang2002algebra} the degree $[L'_0 : K'_0]$ divides $[L_0:K_0] = |G|$, hence $K'_0 = (L'_0)^G$. This proves the claim.
                
                We thus obtain the following diagram of field extensions:
                \[ \xymatrix{   L' \ar@{-}[d]_{\text{$G$-Galois}}  \ar@{-}[r] & L \ar@{-}[d]^{\text{$G$-Galois}}  & & & & L'_0 \ar@{-}[d]_{\text{$G$-Galois}}  \ar@{-}[r] & L_0 \ar@{-}[d]^{\text{$G$-Galois}}  \\          
                        K' \ar@{-}[r] \ar@{-}[d] &                K  & &\text{and} &  & K'_0 \ar@{-}[d]  \ar@{-}[r] & K_0  \\    
                        k & & & & & k_0 & }\]
                where the right side are obtained from the left side by passing to residue fields. In other words,
                the $G$-Galois extension $L_0/K_0$ descends to $K_0'$. By \cite[Lemma 2.1]{brosnan2018essential},
                $\trdeg_{k_0}(K'_0) \leqslant \trdeg_k(K')$. 
                We conclude that
                \[ \ed_{k_0}(L_0/K_0)  \leqslant \trdeg_{k_0}(K'_0) \leqslant \trdeg_k(K') = d = \ed_k(L/K), \]
                as desired.
        \end{proof}     
        
                We are now ready to complete the proof of \Cref{puiseux} assuming (iii).          
                The class $\alpha \in H^1(R, G)$ is represented by a finite \'etale $R$-algebra $E$ 
                such that $G$ acts faithfully on $E$ over $R$ and transitively permutes the connected components of $\Spec (E)$.
                Recall that $K$ is the field of fractions of $R$.
                Since $R$ is complete, it is Henselian, hence every connected component of $\Spec E$ is local, {\it i.e.} is of the form
                $\Spec(E_1)$, where $E_1$ is a discrete valuation ring with valuation $\nu \colon E_1 \setminus \{ 0 \} \to \Z$ 
                extending the valuation on $R$. Denote the fraction field of $E_1$ by $L$ and
                the residue field of $E_1$ by $L_0$. Let us now consider two cases.
                
\smallskip
                {\bf Case 1.} $E$ is integral, {\it i.e.} $E = E_1$. In this case the class $\alpha_K \in H^1(K, G)$ is
                represented by the field extension $L/K$, where $L$ is the fraction field of $E$.
                By~\Cref{prop1b}(a) the class $\alpha_{K_0}$ is represented by 
                the field extension $L_0/K_0$.  The inequality  
                \begin{equation} \label{e.puiseux(iii)} \ed_{k_0}(\alpha_{K_0})\leqslant \ed_{k}(\alpha_{K})
                \end{equation}
                of \Cref{puiseux} now translates to the inequality 
                \[ \ed_{k_0}(L_0/K_0) \leqslant \ed_k(L/K)  \]
                of \Cref{prop1b}(b).

        \smallskip
        {\bf Case 2.} $E_1 \neq E$. To handle this case we will need the following.
        
        \begin{lemma} \label{lem2.1} Let $G$ be a finite discrete group, $H$ be a subgroup of $G$, $K/k$ be a field extension and $\mathcal{T} \to \Spec(K)$ be an $H$-torsor.
                Let $\mathcal{T}_G = \mathcal{T} \times^H G$ be the $G$-torsor induced by $\mathcal{T}$. Then $\ed_k(\mathcal{T}; H) = \ed_k(\mathcal{T} \times^H G; G)$.
        \end{lemma}
        
        \begin{proof} A proof is implicit in \cite[Proposition 2.17]{brosnan2007essential} and is worked out in detail in
        \cite[Lemma 2.4]{bresciani}.
        \end{proof}
        
        We can now complete the proof of \Cref{puiseux} (iii) in Case 2. Denote the stabilizer of $E_1$ in $G$ by $H$. Then
        $\Spec(E) = \Spec(E_1) \times^H G$. In other words, the $G$-torsor $\Spec(E) \to \Spec(R)$ is induced by the $H$-torsor 
        $\Spec(E_1) \to \Spec(R)$ is an $H$-torsor. Replacing $E$ by $E_1$ and $G$ by $H$, and using~\Cref{lem2.1}, we reduce Case 2 to Case 1. This completes the proof of
        \Cref{puiseux}.
        \qed

\smallskip      
We conclude this section with a variant of \Cref{puiseux}, where $A$ and $R$ are not required to be complete.

                        \begin{thm} \label{cor.not-complete}
                Let $A$ be a discrete valuation ring with maximal ideal $\mathfrak{m}$, fraction field $k$, residue field $k_0$ and $\mathfrak{m}$-adic completion $\hat{A}$.
                Set $p := \on{char}(k_0) \geqslant 0$ and let $G$ be a smooth affine group scheme over $A$, satisfying 
                one of the conditions (i), (ii) or (iii) below. Let $R\supset A$ be a discrete valuation ring with fraction field $K\supset k$ and residue field $K_0\supset k_0$, and assume that $\mathfrak{m}$ is contained in the maximal ideal of $R$. Then for every $\alpha\in H^1(R,G)$ we have \[\ed_{k_0}(\alpha_{K_0})\leqslant \ed_{k}(\alpha_{K}).\]
                Furthermore, if $\hat{A}=k_0[[t]]$ and $G_{\hat{A}}$ is defined over $k_0$, then the above inequality is an equality.
                
                        \begin{enumerate}[label=(\roman*)]
                        \item $p = 0$, and there exist a section $\sigma:k_0\to \hat{A}$ of the projection $\hat{A}\to k_0$, and a $k_0$-group $H$ such that $G\simeq \sigma^*H$.
                        \item $G^{\circ}$ is reductive, $G/G^{\circ}$ is $A$-finite, and there exists a finite subgroup $S\subset G(\hat{A})$ such $S$ is tame at $p$ and
                        for every field $L$ containing $k$ the natural map $H^1(L,S)\to H^1(L,G)$ is surjective.
                        \item $G_{\hat{A}} = S_{\hat{A}}$, where $S$ is an abstract finite group which is weakly tame at $p$.
                \end{enumerate}
                \end{thm}
                
                \begin{proof} We have
                        \[\ed_{k_0}(\alpha_{K_0}) \stackrel{(a)}{\leqslant} \ed_{\hat{k}}(\alpha_{\hat{K}}) \stackrel{(b)}{\leqslant} \ed_{k}(\alpha_{\hat{K}}) \stackrel{(c)}{\leqslant}       \ed_k(\alpha_K)\]
                        where $\hat{R}$ is the completion of $R$, and $\hat{k}$, $\hat{K}$ 
                        are the fraction fields of $\hat{A}$, $\hat{R}$, respectively. 
                        Here (a) follows from~\Cref{puiseux}, (b) from \Cref{lem.prel1}(b), and (c) from \Cref{lem.prel1}(a).
       \end{proof}

        \section{Proof of Theorem~\ref{thm.map}}
        \label{sect.proof}
        
        Throughout this section, 
\begin{equation}        \label{e.assumptions}
        \text{$G$ will denote a linear algebraic group and
        $X$ a 
        %quasi-projective 
        $G$-variety},
        \end{equation}
        both defined over a field $k$. We begin with a brief review of the concepts that will be used in the proof.
        
        \subsection*{Twisting}
        For any field $K/k$ and $G$-torsor $\tau \colon \mathcal{T} \to \Spec(K)$, we can define the algebraic space ${\, }^{\tau} X$ over $K$
        as the quotient of $\mathcal{T} \times_k X_K$ by the diagonal action of $G$. Note that in this situation $G$ acts freely on $\mathcal{T} \times_K X_K$, 
        and ${\, }^{\tau} X$ can be defined by descent so that the quotient morphism $\mathcal{T} \times_K X_K \to {\, }^{\tau} X$ is a $G$-torsor.
        If $X$ is quasi-projective, then ${\, }^{\tau} X$ is a $K$-variety and not just an algebraic space. 
        The algebraic space ${\, }^{\tau} X$ is separated and of finite type over $K$. This may be checked after splitting $\tau$;
        when $\tau$ is split, we have ${ \, }^{\tau} X \simeq X_K$, where $\simeq$ denotes isomorphism over $K$.
        In general, ${\, }^{\tau} X$ does not inherit a 
    $G$-action from $X$. There is a natural action of the twisted group ${\, }^{\tau} G$ on ${\, }^{\tau} X$, 
        but we will not need it in the sequel. 
        For a more detailed discussion of the twisting construction and further references, see~\cite[Section 3]{duncan2015versality}.
        
        \subsection*{Large fields} A field $K$ is called {\em large} if the following property holds for every irreducible
        $K$-variety $Z$: if $Z$ has a smooth $K$-point, then $K$-points are dense in $Z$. This notion is due to F.~Pop. 
        Note that if $Z$ is irreducible and has a smooth $K$-point, then it is absolutely irreducible.
        A variant of this definition assumes that $Z$ is an irreducible $K$-curve, rather than an irreducible $K$-variety of
        arbitrary dimension. This a priori weaker property of $K$ turns out to be equivalent. The most important examples of large fields 
        for this paper are fields of Laurent series $K((t))$, where $K$ is an arbitrary field (see~\cite[Section 1A(2)]{pop-large}) and $p$-closed fields (see~\cite[p.~360]{colliot-annals} or~\cite[Theorem 1.3]{pop-large}). See~\Cref{sect.ed-at-p} for the definition of a $p$-closed field and~\cite{pop-large} for a detailed discussion of large fields, including further examples.

        \begin{lemma}\label{large-algebraic-space}
Let $K$ be a large field and $Z$ be an irreducible separated algebraic space of finite type over $K$. If $Z$ has a smooth $K$-point, then $K$-points are dense in $Z$.
\end{lemma}

\begin{proof}
    Let $z\in Z(K)$ be a smooth point, and let $Z':=\on{Proj}_Z(\oplus_{i\geqslant 0} I^i)\to Z$ be the blow-up of $Z$ at $z$, where  $I$ is the ideal sheaf of $z$; see \cite[Tag 085Q]{stacks-project}. It follows that the exceptional divisor $E$ of $Z'\to Z$ is given by $\on{Proj}_{z}(\oplus_{i\geqslant 0} I^i/I^{i+1})=\P(T_{Z,z})$, where $T_{Z,z}$ denotes the tangent space at $z$. Since $K$ is large, it is infinite, hence $K$-points are dense in $E$. Since $z$ is smooth, any $K$-point in $E$ is smooth. By \cite[Tag 0ADD]{stacks-project}, there exists an open embedding $U\hookrightarrow Z'$ such that $U$ is a scheme and $U\cap E\neq \emptyset$. Thus $U$ contains smooth $K$-points and so, since $K$ is large and $U$ is irreducible, $K$-points are dense in $U$. Since the blow-up map $Z'\to Z$ is surjective, we conclude that $K$-points are dense in $Z$, as desired.
\end{proof}

        \subsection*{$\tau$-versality}
        Suppose $K/k$ is a field extension and $\tau \colon \mathcal{T} \to \Spec(K)$ is a $G$-torsor. We say that $X$ is $\tau$-weakly versal 
        if there exists a $G$-equivariant morphism $\mathcal{T} \to X_K$ defined over $K$. 
        We say that $X$ is $\tau$-versal if every dense $G$-invariant open subvariety of $X$ is $\tau$-weakly versal.
        Note that if $X$ is primitive, then every non-empty $G$-invariant open subvariety is automatically dense.

\smallskip
With these preliminaries out of the way, we can get started on the proof of \Cref{thm.map}.

        \begin{lemma} \label{lem.versality1} Let $G$ and $X$ be as in~\eqref{e.assumptions}, 
                $K/k$ be a field extension, and $\tau \colon \mathcal{T} \to \Spec(K)$ be a $G$-torsor. 
                
                \smallskip
                (a) $X$ is $\tau$-weakly versal if and only if the twisted variety ${\, }^\tau \! X$ has a $K$-point.
                
                \smallskip
                (b) Assume further that $X$ is smooth, and $K$ is a large field. 
                Then $X$ is weakly $\tau$-versal if and only if $X$ is $\tau$-versal.
        \end{lemma}
        
        \begin{proof} Part (a) is proved at the beginning of Section 4 in~\cite{duncan2015versality}. Note that the proof relies on
                \cite[Proposition 3.2]{duncan2015versality}, where $X$ is not assumed to be irreducible.
                
                (b) If $X$ is $\tau$-versal, then $X$ is obviously weakly $\tau$-versal. Conversely, suppose $X$ is weakly $\tau$-versal. We want to show that
                $X$ is $\tau$-versal. Assume the contrary: there exists a Zariski dense open $G$-invariant subvariety $U$ of $X$ which 
                is not weakly $\tau$-versal. Then ${\, }^\tau \! U$ is a dense open algebraic subspace of ${\, }^\tau \! X$ 
                (see \cite[Corollary 3.4]{duncan2015versality}). By part (a), ${\, }^\tau \! U$ has no $K$-points but
                ${\, }^\tau \! X$ has a $K$-point. Denote this point by $x$. 
                Since $X$ is smooth, so is ${\, }^{\tau} X$. Hence, $x$ lies on a unique irreducible component of 
                ${\, }^\tau \! X$. Denote this irreducible component by $X_1$. Since $X_1$ is smooth and irreducible 
                and $K$ is large, $K$-points are dense in ${\, }^\tau \! X_1$. In particular, there is a $K$-point in
                ${\, }^\tau \! U \cap X_1$, contradicting our assumption that ${\, }^\tau \! U$ has no $K$-points.
        \end{proof}

        \begin{lemma} \label{lem.versality2} Let $G$ and $X$ be as in~\eqref{e.assumptions}.
                Assume that $X$ is generically free and primitive, $K/k$ is a field extension and 
                $\tau \colon \mathcal{T} \to \Spec(K)$ is a $G$-torsor. If $X$ is $\tau$-versal, then $\ed_k(X) \geqslant \ed_k(\tau)$.
        \end{lemma}
        
        \begin{proof}
                Let $f \colon X \dashrightarrow Z$ be a $G$-compression ({\it i.e.} a $G$-equivariant dominant rational map defined over $k$) to a generically free
                $G$-variety $Z$ of minimal possible dimension, $\dim(Z) = \ed_k(X) + \dim(G)$.
                As we explained in Section~\ref{sect.ed-variety}, after replacing $Z$ by a dense open $G$-invariant subvariety we may assume that $Z$ is the total space of a $G$-torsor $Z \to B$,
                over some $k$-variety $B$ of dimension 
                \[ \dim(B) = \dim(Z) - \dim(G) = \ed_k(X). \]                 
                Let $X_0 \subset X$ be the domain of $f$. 
                Since $X$ is $\tau$-versal, there exists a $G$-equivariant map $\mathcal{T} \to (X_0)_K = 
                X_0 \times_{\Spec(k)} \Spec(K)$
                defined over $K$. Projecting to the first component and composing with $f$, we 
                obtain a $G$-equivariant morphism
                \[  \mathcal{T} \to % \overline{f \circ j (\mathcal{T})} \subset
                Z \]
                defined over $k$. This morphism induces a pull-back diagram 
                 \[ \begin{gathered}
        \xymatrix{   \mathcal{T} \ar@{->}[d]_{\tau}  \ar@{->}[r] & Z \ar@{->}[d]  \\          
                               \Spec(K) \ar@{->}[r]    & B } \end{gathered}
        \]
                of $G$-torsors. Restricting to the generic point of the image of $\Spec(K)$ in $B$ (whose residue field we will denote by $K_0$), we obtain a pull-back diagram
                of $G$-torsors
                 \[ \begin{gathered}
        \xymatrix{   \mathcal{T} \ar@{->}[d]_{\tau}  \ar@{->}[r] & \mathcal{T}_0 \ar@{->}[d]^{\tau_0}   \\          
                               \Spec(K) \ar@{->}[r]    & \Spec(K_0) ,} \end{gathered}
        \]
        as in~\eqref{e.compression}.
                By definition of $\ed_k(\tau)$, 
                $\ed_k(X) = \dim(B) \geqslant \trdeg_k(K_0) \geqslant \ed_k(\tau)$,
                as desired.
        \end{proof}
        
     Note that~\Cref{lem.versality2} is in the same general spirit as~\Cref{thm.map}. Indeed, let $\tau = \tau_Y \colon \mathcal{T} \to \Spec(K)$ 
     be the $G$-torsor over $K = k(Y)^G$ constructed from $Y$, as in \Cref{sect.ed-variety}. 
        Then $\ed_k(\tau_Y) = \ed_k(Y)$, and the existence of a $G$-equivariant rational map
        $f \colon Y \dasharrow X$ is equivalent to saying that $X$ is weakly $\tau$-versal. 
        \Cref{lem.versality2} requires $X$ to be $\tau$-versal, which is a stronger assumption in general.
        To bridge the gap between ``weakly $\tau$-versal" and ``$\tau$-versal", we will use the fact that these 
        two notions are equivalent when $K$ is a large field; see \Cref{lem.versality1}(b).
        Our strategy is thus to replace $k$ by $k((t))$, $K$ by $K((t))$ and $\tau$ by $\tau((t)) = \tau \otimes_K K((t))$.
        The following lemma tells us that under assumptions (1) -- (4) on $G$ of \Cref{thm.map}, 
        passing from $\tau$ to $\tau((t))$ does not change the essential dimension.
        
        \begin{lemma}\label{red-str-conn-red} Let $G$ be a linear algebraic group over $k$, satisfying one of the conditions (1) -- (4) of \Cref{thm.map}.
         Then for $A = k[[t]]$ the group scheme $G_A$ satisfies one of the conditions (i) -- (iii) of \Cref{puiseux}.
        \end{lemma}
        
        \begin{proof} In the proof (i), (ii), (iii) will always refer to conditions of \Cref{puiseux} and (1), (2), (3), (4) to
        conditions of \Cref{thm.map}. The following implications are obvious: 
\[ G \ \text{satisfies (1)} \Longrightarrow G_A \ \text{satisfies (i)}\quad \text{and}\quad G \ \text{satisfies (4)} \Longrightarrow G_A \ \text{satisfies (iii)}.\] 
The implication 
\[G\ \text{satisfies (2)} \Longrightarrow G_A\ \text{satisfies (ii)}\]
follows from \Cref{red-structure-split}. The implication 
\[G \ \text{satisfies (3)} \Longrightarrow G_A \ \text{satisfies (ii)}\]
follows from~\cite[Theorem 1.1(c)]{chernousov2006resolving} and~\cite[Remark 4.1]{chernousov2008reduction}.
\end{proof}
        
        We are now ready to complete the proof of \Cref{thm.map}. Let $K = k(Y)^G$ and
                $\tau = \tau_Y \colon \mathcal{T} \to \Spec(K)$ be the $G$-torsor constructed from $Y$, as above.
                The $G$-equivariant rational map $f \colon Y \dasharrow X$ tells us that $X$ is weakly $\tau$-versal. 
                Let $L := K \otimes_k k((t))$. Then the inclusions of fields
                \[ k \hookrightarrow k((t)) \hookrightarrow L \hookrightarrow K((t)) \]
                induce a Cartesian diagram
                \begin{equation}\label{diag-large} \xymatrix{   \mathcal{T}_{K((t))} \ar@{->}[d]_{\tau((t))}  \ar@{->}[r] & \mathcal{T}_{L} \ar@{->}[d]^{\tau_1} \ar@{->}[r] & \mathcal{T} \ar@{->}[d]^{\tau} \ar@{->}[r]^f   & X_K   \\
                                \Spec(K((t))) \ar@{->}[r] &                \Spec(L) \ar@{->}[r] & \Spec(K).   &  }  
                \end{equation}
                Let $X_1 = X_{k((t))}= X \times_{\Spec(k)} \Spec(k((t)))$
                viewed as a $G$-variety over $k((t))$. (Here $G$ acts on the first factor.)
                The $G$-equivariant map $f \colon \mathcal{T} \to X$ defined over $k$ naturally extends to a $G$-equivariant map 
                $\mathcal{T}_{K((t))} \to X_{K(t))} = (X_1)_{K((t))}$ defined over $K((t))$. 
                This shows that $X_1$ is weakly $\tau((t))$-versal. Since $K((t))$ is a large field,
                \Cref{lem.versality1}(b) tells us that $X_1$ is $\tau((t))$-versal. Now observe that
                \begin{equation} \label{e.map-large1} \ed_k(X) \stackrel{(a)}{\geqslant} \ed_{k((t))} (X_1) \stackrel{(b)}{\geqslant} \ed_{k((t))}\Big ( \tau((t)) \Big) \stackrel{(c)}{=} \ed_k(\tau) \stackrel{(d)}{=} \ed_k(Y). \end{equation}
                Here (a) follows from \Cref{lem.prel1.5}(a),
                (b) follows from \Cref{lem.versality2}, and (d) is a restatement of~\eqref{e.torsor-variety}.
                Finally, (c) follows from \Cref{puiseux}. Note that~\Cref{puiseux} can be applied in this situation
                under any of the assumptions (1) -- (4) of \Cref{thm.map} by~\Cref{red-str-conn-red}.
        \qed
        
        \section{Some consequences of Theorem~\ref{thm.map}}
        \label{sect.cor}
        
        Throughout this section $k$, $G$ and $X$ will be as in~\eqref{e.assumptions}.   
        Recall that $X$ is called weakly versal (respectively, versal) if it is $\tau$-weakly versal (respectively,
        $\tau$-versal) for every field $K/k$ and every $G$-torsor $\tau \colon \mathcal{T} \to \Spec(K)$; see~\cite[p.~499]{duncan2015versality}. 
        
        \begin{cor} \label{cor.weakly-versal}
                Assume that $G$ satisfies one of the conditions (1) -- (4) of \Cref{thm.map}
                and the $G$-variety $X$ is 
                %quasi-projective, 
                generically free and primitive.
                
                \smallskip
                (a) If $X$ is smooth and weakly versal, then $\ed_k(X) = \ed_k(G)$.
                
                \smallskip
                (b) If $X$ has a smooth $G$-fixed $k$-point, then $\ed_k(X) = \ed_k(G)$.
        \end{cor}
        
        \begin{proof} (a) Let $V$ be a generically free linear representation of $G$ and $\tau_V \colon \mathcal{T} \to \Spec(K)$ be
        the $G$-torsor over $K = k(V)^G$ constructed from $V$, as in~\Cref{sect.ed-variety}. Since $X$ is $\tau_V$-versal, there exists
         a $G$-equivariant rational map $V \dasharrow X$ defined over $k$.
                By \Cref{thm.map}, $\ed(X) \geqslant \ed(V)$. 
                On the other hand, by \Cref{lem.prel2}(a),      $\ed(V) = \ed(G) \geqslant \ed(X)$ and part (a) follows.
                
                \smallskip
                (b) After replacing $X$ by its smooth locus, we may assume that $X$ is smooth. The constant map $\mathcal{T} \to X_K$
                sending $\mathcal{T}$ to the $G$-fixed point in $X$ is $G$-equivariant for every $G$-torsor $\tau \colon \mathcal{T} \to \Spec(K)$. 
                Thus $X$ is weakly versal and part (a) applies.
        \end{proof}
        
        \begin{example} \label{ex.product}
                Let $k$ be a field of characteristic $p \geqslant 0$ and $G$ be a subgroup of the symmetric group $\Sym_m$. Assume that $G$ is weakly versal at $p$. 
                For any absolutely irreducible variety $Z$ defined over $k$, we may view $Z^m$ as a $G$-variety, where $G$ acts on $Z^m$ by permuting the factors.
                Then \[ \ed_{k}(Z^m)=\ed_k(G). \] 
        Indeed, if $z$ is a smooth point of $Z$, then $(z, \ldots, z)$ is a smooth $G$-fixed point of $Z^m$, and \Cref{cor.weakly-versal}(b) applies.
                \qed    \end{example}
        
        \begin{example} \label{ex.connected-group} Let $G$ be a linear algebraic group satisfying one of the conditions (1) -- (4) of \Cref{thm.map}, and let $X$ be a smooth connected algebraic group defined over a field $k$ (not necessarily linear).
                Suppose $G$ acts generically freely on $X$ by group automorphisms. Then 
                \begin{equation} \label{e.connected-group}
                        \ed_k(X) = \ed_k(G).
                \end{equation}
                Indeed, the identity element of $X$ is a smooth fixed point of $G$, and \Cref{cor.weakly-versal}(b) applies. 
                
                In the case, where $X$ is affine, this was previously known. In this case $X$ is versal as a $G$-variety (see~\cite[Example 7.4]{duncan2015versality}), and~\eqref{e.connected-group} follows. If $X$ is not affine, it is no longer versal as a $G$-variety 
                (only weakly versal) but equality~\eqref{e.connected-group} holds nevertheless.
        \end{example}
        
        \begin{cor} \label{cor.product} 
        Let $G$ be a linear algebraic group satisfying one of the conditions (1) -- (4) of \Cref{thm.map}, and let $X$ and $Y$ be $G$-varieties. 
                Assume that $Y$ is generically free and primitive, and that $X$ is smooth and absolutely irreducible. 
                If there exists a $G$-equivariant rational map $f \colon Y \dasharrow X$, then 
                \begin{equation} \label{e.product}
                        \ed_k(X \times_k Y) = \ed_k(Y).
                \end{equation}
        \end{cor}
        
        \begin{proof} Consider the $G$-equivariant maps (i) $f \times_k \text{id}  \colon Y \dasharrow X \times_k Y$ and 
                (ii) $\text{pr}_2 \colon X \times_k Y \to Y$.
                By \Cref{thm.map}, (i) implies that $\ed_k(X \times_k Y) \geqslant \ed_k(Y)$. Similarly, (ii) 
                implies that
                $\ed_k(Y) \geqslant \ed_k(X \times_k Y)$, and Equality~\eqref{e.product} follows. 
        \end{proof}

        Note that the $G$-action on $X$ in the statement of \Cref{cor.product} is not assumed to be generically free.
        As a special case, \eqref{e.product} holds if $G$ acts trivially on $X$ and $X(k)\neq \emptyset$. 
        This was previously known only in the case where $k$-points are dense in $X$.

        \section{Counterexamples}
        \label{sect.counterexamples}
        
        \begin{lemma} \label{lem.versality3} Let $G$ be a linear algebraic group defined over $k$, $K$ be a large field containing $k$ 
                and $\tau \colon \mathcal{T} \to \Spec(K)$ be a $G$-torsor.
\begin{enumerate}[label=(\alph*)]
\item Let $X$ be a generically free $G$-variety defined over $k$. If $X$ has a smooth $G$-fixed $k$-point, then $\ed_k(\tau) \leqslant \ed_k(X) \leqslant \dim(X) - \dim(G)$. 
\item If $\Char(k) = p$ and $G$ is a finite $($constant$\,)$ $p$-group, then $\ed_k(\tau) \leqslant 1$. In other words, $\ed_k(L/K) \leqslant 1$ for every $G$-Galois field extension $L/K$.
\item If $\Char(k) = p$ and $G$ is a finite $p$-group which is not elementary Abelian, then~\Cref{thm.map}, \Cref{cor.weakly-versal}, \Cref{cor.product} and \Cref{puiseux} all fail for $G$.
\end{enumerate}
        \end{lemma}
        
        \begin{proof} (a) After replacing $X$ by its smooth locus, we may assume that $X$ is smooth. The constant map $\mathcal{T} \to X_K$ sending $\mathcal{T}$ to the fixed point shows that
                $X$ is weakly $\tau$-versal. By \Cref{lem.versality1}(b), $X$ is $\tau$-versal. The inequality  $\ed_k(\tau) \leqslant \ed_k(X)$ now follows from \Cref{lem.versality2}. The inequality $\ed_k(X) \leqslant \dim(X) - \dim(G)$ follows from the definition of essential dimension.
                
                \smallskip
                (b) In this case there exists a smooth irreducible $G$-curve $X$ defined over $k$ 
                such that $G$ acts generically freely ({\it i.e.} faithfully) on $X$ and fixes a $k$-point;
                see~\cite[Lemma 3]{reichstein-vistoli-prime}. By part (a), 
                $\ed_k(\tau) \leqslant \dim(X) - \dim(G) = 1$.
                
                \smallskip
                (c) By~\cite[Proposition 5]{ledet-ed1}, $\ed_k(G) \geqslant 2$. By \Cref{lem.prel2}(a), this tells us that 
                \begin{equation} \label{e.geq2}
                        \ed_k(V) \geqslant 2
                \end{equation}
                where $V$ is the regular representation of $G$ over $k$. Let $X$ be a smooth $G$-curve with a fixed $k$-point, as in part (b), 
                and $f \colon V \to X$ be the constant $G$-equivariant morphism sending $V$ to this $G$-fixed point. 
                If \Cref{thm.map} were true in this case,
                then we would have
                \[ \ed_k(V) \leqslant \ed_k(X) \leqslant \dim(X) = 1, \]
                contradicting~\eqref{e.geq2}. This shows that \Cref{thm.map} fails for $G$.
                
                If~\Cref{cor.weakly-versal} were valid for $X$, we would conclude that $\ed_k(X) = \ed_k(G)$. 
                This leads to a contradiction because $\ed_k(G) \geqslant 2$ but $\ed_k(X) \leqslant 1$. 
                
                Since the $G$-action on both $X$ and $V$ is generically free, and both have smooth $G$-fixed points,
                there exists $G$-equivariant maps $X \to V$ and $V \to X$. If
                \Cref{cor.product} were true in this setting, it would tell us that
                \[ \ed_k(V) = \ed_k(V \times_k X) = \ed_k(X \times_k V) = \ed_k(X) \leqslant 1, \]
                contradicting~\eqref{e.geq2}.  Thus \Cref{cor.product} fails in this setting as well.
                
                Finally, to show that \Cref{puiseux} fails, let $K = k(V)^G$ and $\tau \colon \mathcal{T} \to \Spec(K)$ be the $G$-torsor associated to
                the $G$-action on $V$, as in \Cref{sect.ed-variety}. Then $\ed_k(\tau) = \ed_k(V) \geqslant 2$.
                On the other hand, since $K((t))$ is a large field, part (b) tells us that $\ed_{k((t))} \big( \tau((t)) \big) \leqslant 1$. In particular,
                $\ed_{k((t))} \big( \tau((t)) \big) < \ed_k(\tau)$. This shows that \Cref{puiseux} fails for $A = k[[t]]$, $R = K[[t]]$ and $\alpha = \tau_{R}$. (Note that none of the conditions (i), (ii), (iii) of
                \Cref{puiseux} are satisfied here.)
        \end{proof} 
        
        \begin{rmk} \label{rem.p-closed} \Cref{lem.versality3}(b) is proved in~\cite{reichstein-vistoli-prime} 
                in the case where $K$ is a $p$-closed field. Recall that $p$-closed fields are large; see~\cite[p.~360]{colliot-annals}.
        \end{rmk} 
        
        \begin{rmk} A conjecture of A.~Ledet~\cite{ledet-p} asserts that $\ed_k(\Z/ p^n \Z) = n$ for any infinite base field $k$ of characteristic $p$ 
                and any integer $n \geqslant 1$. This conjecture has been proved for $n \leqslant 2$ and is open for every $n \geqslant 3$. Moreover,
                Ledet showed that $\ed_k(\Z/ p^n \Z) \leqslant n$, so his conjecture is equivalent 
                to the existence of a field extension $K/k$ and a $\Z/ p^n \Z$-torsor $\tau \colon \mathcal{T} \to \Spec(K)$ 
                such that $\ed_k(\tau) \geqslant n$. 
                
                This conjecture has many interesting consequences~\cite{brosnan2018essential}, \cite{tossici-unipotent}.
                It looks even more remarkable in light of~\Cref{lem.versality3}(b), which asserts that
            $\ed_{k((t))} \big( \tau((t)) \big) \leqslant 1$ for every
                field $K$ containing $k$ and every $G$-torsor $\tau$ over $\Spec(K)$.
        \end{rmk}
        
        \begin{rmk} The smoothness assumption on $X$ in the statement of~\Cref{thm.map} cannot be dropped.
        \label{rem.smoothness}
        
        Indeed, let $G$ be a finite group. Then there exists a curve $C$ of genus $\geqslant 2$ with a faithful $G$-action. 
    To construct $C$, start with a curve $C_n$ with a faithful $\Sym_n$-action as in~\cite[Remark 4.5]{buhler1997essential}.
        For large $n$, $C_n$ will automatically be of genus $\geqslant 2$, since $\Sym_n$ cannot act faithfully on a rational or elliptic curve.
        After embedding $G$ into $\Sym_n$ (again, for a suitably large $n$), we may view $C = C_n$ as a $G$-curve.
        Next we embed $C$ into a projective space $\mathbb P(V)$ in an $\Sym_n$-equivariant way, where $\Sym_n$ acts linearly on $V$ 
        ({\it e.g.} by using a pluri-canonical embedding) and we consider $X$ the affine cone over $C$. 
        The origin $0$ in $X$ is fixed by $\Sym_n$ (and hence, by $G$). 
        Note that since the center of $\Sym_n$ is trivial for any $n \geqslant 2$, we may assume that 
        no element of $\Sym_n$ (and hence, of $G$) acts by scalar multiplication on $V$.
        
%    We claim that the cone $X \subset V$ over $C \subset \mathbb P(V)$ 
%    is singular at $0$ if $n \geqslant 5$. Indeed, assume the contrary. Since $\dim(X) = 2$ and $0$ is a 
%       smooth point of $X$, $\Sym_n$ acts linearly on the 2-dimensional tangent space $T_0(X)$. For $n \geqslant 5$,
%       $\Sym_n$ does not have a faithful 2-dimensional linear representation. From this we conclude 
%       that the $\Sym_n$-action on $X$ is not faithful, and thus neither is the $\Sym_n$-action on $C$, a contradiction.
%       This proves the claim.
        
   The natural projection $X \dasharrow C$ tells us that $\ed_k(X) \leqslant \dim(C) - \dim(G) = 1$.
        We claim that~\Cref{thm.map} fails for the constant map $f \colon V \to X$ sending all of $V$ to the origin.
        Here $V$ is a faithful linear representation of $G$ over $k$ ({\it e.g.} the regular representation). 
        Indeed, if \Cref{thm.map} were true for $f$, it would tell us that $1 \geqslant \ed_k(X) \geqslant \ed_k(V) = \ed_k(G)$ for every
        irreducible $G$-variety $Y$, {\it i.e.} $1 \geqslant \ed_k(G)$, which is false for most finite groups;
        see~\cite[Theorem 6.2]{buhler1997essential} or~\cite{ledet-ed1}. Of course, the reason for this failure of \Cref{thm.map} 
        in this example is that $X$ is singular at the origin.
        
        Similar (even easier) examples can be constructed for a connected group $G$ as follows. Start with a generically free representation $G \to \GL(V)$, let
        $w$ be a point in general position in $\mathbb P(W)$, where $W = V \oplus k$, and $X$ be the cone over the Zariski closure of the orbit $G \cdot w$ 
        in $\mathbb P(W)$. Clearly $\ed_k(X) = 0$. Once again, 
        the ``constant" morphism $f \colon V \to X$ taking the whole of $V$ to $0$ is $G$-equivariant. 
        If \Cref{thm.map} were true for $f$, then we would obtain $0 \geqslant \ed_k(V) = \ed_k(G)$, which is false for most connected groups $G$,
        {\it e.g.} for $G = \on{SO}_n$ ($n \geqslant 3$) or $\on{PGL}_n$ ($n \geqslant 2$). Once again the reason \Cref{thm.map} fails in this example is that
        $X$ is singular at the origin.
        \qed
        \end{rmk}
        
    \section{An analogue of Theorem~\ref{thm.map} for essential dimension at a prime}
    
        \begin{thm} \label{thm.map-at-p} Let $k$ be a field, $G$ be a linear algebraic group defined over $k$
                and $X$, $Y$ be generically free primitive $G$-varieties defined over $k$. Assume that $X$ is smooth. 
                If there exists a $G$-equivariant correspondence $f \colon Y \rightsquigarrow X$ of degree prime to $q$, 
                then $\ed_{k,q}(X) \geqslant \ed_{k,q}(Y)$.
        \end{thm}
        
        The proof of \Cref{thm.map-at-p} is similar to the proof of \Cref{thm.map}, but simpler. 
        Since the $q$-closure $K^{(q)}$ of any field $K$ is large, we do not need to pass to $k((t))$ or appeal to \Cref{puiseux}. 
        Note that we do not require $G$ to be smooth or reductive here, and $\Char(k)$ can be arbitrary.
        
        \begin{proof}
                Let $K/k$ be a field extension and
                $\tau = \tau_Y \colon \mathcal{T} \to \Spec(K)$ be a $G$-torsor constructed from $Y$, as in \Cref{sect.ed-variety}.
                The correspondence $Y \rightsquigarrow X$ of degree prime to $q$ gives rise to a $G$-equivariant rational map 
                $\tau_{K^{(q)}} \colon \mathcal{T}_{K^{(q)}} \to X$, where
                $K^{(q)}$ is the $q$-closure of $K$. In other words, $X$ is weakly $\tau_{K^{(q)}}$-versal.
                Since $K^{(q)}$ is a large field, \Cref{lem.versality1} tells us that $X$ is $\tau_{K^{(q)}}$-versal.
                By \Cref{lem.versality2},
                \begin{equation} \label{e.at-p} \ed_k(X) \geqslant \ed_k(\tau_{K^{(q)}}) = \ed_{k,q}(\tau) = \ed_{k, q}(Y), \end{equation}
                where the first equality is~\eqref{e.ed-at-p2} and the second equality is~\eqref{e.torsor-variety-at-p}.
                Moreover, suppose $\pi \colon X' \dasharrow X$ is a $G$-equivariant dominant rational map of $G$-varieties of degree prime to $q$.
                Viewing the inverse of $\pi$ as a correspondence $\pi^{-1} \colon X \rightsquigarrow X'$ and composing it with $f$, we obtain 
                a correspondence $\pi^{-1} \circ f \colon Y \rightsquigarrow X'$ of degree prime to $q$. By~\eqref{e.at-p}, we get 
                \[ \ed_k(X') \geqslant \ed_{k, q}(Y). \]
                Taking the minimum over all $G$-equivariant rational covers $X' \dasharrow X$ of degree prime to $q$ and using~\eqref{e.ed-at-p1},
                we arrive at the desired inequality $\ed_{k, q}(X) \geqslant \ed_{k, q}(Y)$.
        \end{proof}

     Using \Cref{thm.map-at-p}, we can deduce prime-to-$q$ analogues of all of the corollaries and examples in \Cref{sect.cor}.
        Note however, that \Cref{cor.weakly-versal} is less novel in this context. In particular, by~\cite[Theorem~8.3]{duncan2015versality} 
        a weakly $q$-versal smooth $G$-variety $X$ is $q$-versal. If $X$ is generically free and primitive, this implies $\ed_{k, q}(X) = \ed_{k, q}(G)$. 
        The case where $X$ has a fixed point is considered in \cite[Corollary 8.6]{duncan2015versality}.
        The equality $\ed_k(X) = \ed_k(G)$ of \Cref{cor.weakly-versal} is more intricate, because $X$ may not be versal. 
        
        On the other hand, to the best of our knowledge, the following prime-to-$q$ analogue of \Cref{cor.product} is new. 
        It is proved by the same argument as \Cref{cor.product}, with \Cref{thm.map-at-p} used in place of \Cref{thm.map}.
        
        \begin{cor} \label{cor.product-at-p} Let $X$ be $Y$ be $G$-varieties. 
                Assume that $Y$ is generically free and primitive, and $X$ is smooth and absolutely irreducible. 
                If there exists a $G$-equivariant correspondence $f \colon Y \rightsquigarrow X$ of degree prime to $q$, then
                $\ed_{k, q}(X \times_k Y) = \ed_{k, q}(Y)$.
                \qed
        \end{cor}       

\appendix

\section*{Appendix A. Essential dimension at a prime}
\addcontentsline{toc}{section}{Appendix A. Essential dimension at a prime}
\refstepcounter{section}        
        
        In this appendix we will prove a variant of Theorem~\ref{puiseux}, where essential dimension is replaced by essential dimension at a fixed prime $q$.
        Note that conditions (i), (ii) and (iii) of \Cref{puiseux} simplify in this setting; they  are replaced by conditions (i') and (ii') below.

        \begin{thm}\label{puiseux'}
                Let $A$ be a complete discrete valuation ring with maximal ideal $\mathfrak{m}$, fraction field $k$ and residue field $k_0$. Set $p := \on{char}(k_0) \geqslant 0$, $q\neq p$ be a prime number, and let $G$ be a smooth affine group scheme over $A$, satisfying one of the conditions (i') or (ii') below.
                
                Let $R\supset A$ be a complete discrete valuation ring with fraction field $K\supset k$ and residue field $K_0\supset k_0$, and assume that $\mathfrak{m}R$ is the maximal ideal of $R$. Then for every $\alpha\in H^1(R,G)$ we have \[\ed_{k_0,q}(\alpha_{K_0};G_{K_0})\leqslant \ed_{k,q}(\alpha_{K};G_{K}).\]
                Furthermore, if $A=k_0[[t]]$ and $G_{A}$ is defined over $k_0$, then the above inequality is an equality.
                
        \begin{enumerate}[label=(\roman*')]
                        \item $\on{char}(k_0)=0$, and there exist a section $\sigma:k_0\to A$ of the projection $A\to k_0$ and a $k_0$-group $H$ such that $G\simeq \sigma^*H$;
                        \item $G^{\circ}$ is reductive, there exists a finite subgroup $S\subset G(A)$ of order invertible in $k_0$ such that for every $q$-closed field $L$ containing $k$ the natural map $H^1(L,S)\to H^1(L,G)$ is surjective.
        %               \item $G_{A}$ is the constant $A$-group scheme associated to a finite abstract group without non-trivial normal $p$-subgroups, 
        %        where $p=\on{char}(k_0)$.
                \end{enumerate} 
        \end{thm}
        
        Note that if $S$ is the finite group in (ii') and $S_q$ is the Sylow $q$-subgroup of $S$, then $H^1(L, S_q) \to H^1(L, S)$ is an isomorphism
        for every $q$-closed field $L$. Hence, we may replace $S$ by $S_q$ in (ii'). In other words, if the finite subgroup group $S$ of $G(A)$
        in (ii') exists, we may assume that $S$ is a $q$-group.
        
                Our proof of \Cref{puiseux'} below is analogous to that of \Cref{puiseux} in~\Cref{sect.puiseux(i)-(ii)}. However, since the definition of $\ed_{k,q}(\alpha_{K};G_{K})$ allows replacing $K$ by a finite extension $L/K$ of degree prime to $q$ (and similarly for $\ed_{k_0,q}(\alpha_{K_0};G_{K_0})$)  
                % of essential $q$-dimension involves a finite extension of degree prime to $q$, 
                there are some complications involving extensions of valuations from $K$ to $L$.
                % along such finite extensions.
                % and the choice of roots of the uniformizer. 
                % We will therefore spell out the proof in detail.  

        \begin{proof}[Proof of \Cref{puiseux'}] 
     Let $\pi\in A$ be a uniformizer. Using~\Cref{e=1}, as we did at the beginning of the proof of \Cref{puiseux} in Section~\ref{sect.puiseux(i)-(ii)}, we may assume that $\mathfrak{m}R$ is the maximal ideal of $R$, that is, that $\pi$ is also a uniformizer in $A$. Let $K'/K$ be a finite extension of degree prime to $q$ and $k\subset F\subset K'$ be a field of definition for $\alpha_{K'}$ such that
                \begin{equation}\label{e.ed-alphat'}
                        \ed_{k,q}\alpha_K=\trdeg_{k}F.
                \end{equation}
                Let $\displaystyle v \colon (K')^{\times}\to \dfrac{1}{d}\Z$
                be the unique surjective valuation extending the given valuation on $K$. Let $R'\subset K'$ be the local ring and $K'_0$ the residue field of $v$. Since $q$ does not divide $[K':K]$, it does not divide $d$. %We have $(\pi')^d=\lambda\pi$ for some $\lambda\in R^{\times}$. Replacing $\pi$ by $\lambda\pi$, we may assume that $(\pi')^d=\pi$.
                Moreover, since $F$ contains $k$ and $\pi\in A\subset k$, we have
                \[\Z\subset v({F^{\times}})\subset \frac{1}{d}\Z.\] 
                Let $O\subset F$ be the valuation ring and $F_0\subset K_0'$ the residue field of $v|_{F^{\times}}$. The inclusion $A\subset O$ of valuation rings induces an inclusion $k_0\subset F_0$ of residue fields. By \cite[Lemma 7.1]{reichstein2020essential}, we have
                \begin{equation} \label{e.residue-field'} \trdeg_{k}F\geqslant \trdeg_{k_0}F_0.
                \end{equation}
                Let $\hat{R}$ be the $\pi$-adic completion of $R$, and let $\hat{F}$ be the fraction field of $\hat{R}$, that is, the completion of $F$ as a valued field. Since $K'$ is complete, the universal property of the completion gives a unique field embedding $\hat{F}\hookrightarrow K'$ extending the inclusion $F\subset K'$. We view $\hat{F}$ as a subfield of $K'$ via this embedding, so that $F\subset \hat{F}\subset K'$, $O\subset \hat{O}\subset R'$, and the residue field of $\hat{F}$ is $F_0$.

                If $L$ is a field containing $k$ and $n\geqslant1$ is an integer invertible in $k$, we write $L_n$ for the \'etale $L$-algebra $L[x]/(x^n-\pi)$. Then $L_n$ factors as a product of finite-dimensional separable field extensions of $L$, each of which contains an $n^\mathrm{th}$ root of $\pi$. Moreover, if $L$ is the fraction field of a complete discrete valuation ring with uniformizer $\pi$, then by Eisenstein's Criterion $L_n$ is a field. Thus $k_n$, $\hat{F}_n$, and $K_n$ are fields, while $K_n'$ is not necessarily a field. Since $k_n$ is a field, the inclusion $k_n\subset K_n'$ factors through a unique field factor $K_n''\subset K_n'$. Thus $k_n$, $\hat{F}_n$ and $K_n$ are all contained in $K_n''$.
                
                We have a commutative diagram 
                \[
                \begin{tikzcd}
                        F_0 \arrow[d,hook] & \arrow[l]  \hat{O} \arrow[d,hook]\arrow[r,hook] &   \arrow[r,hook] \hat{F} \arrow[d,hook] & \hat{F}_n \arrow[d, hook]  \\
                        K_0' & \arrow[l] R' \arrow[r,hook] & K' \arrow[r,hook] & K''_n \\
                        K_0 \arrow[u, hook]  & \arrow[l]  R \arrow[u, hook] \arrow[r,hook]  & K \arrow[u, hook] \arrow[r, hook]  & K_n. \arrow[u, hook]
                \end{tikzcd}
                \]
                Passing to Galois cohomology, we obtain the following commutative diagram:
                \begin{equation}\label{key-diag'}
                        \begin{tikzcd}
                                H^1(F_0,G) \arrow[d] & \arrow[l,swap, "\sim"]  H^1(\hat{O},G) \arrow[d]\arrow[r,hook] & H^1(\hat{F}_n,G) \arrow[d]  \\
                                H^1(K'_0,G)  & \arrow[l,swap, "\sim"] H^1(R',G)  \arrow[r,hook] & H^1(K''_n,G)\\
                                H^1(K_0,G) \arrow[u] & \arrow[l,swap,"\sim"] H^1(R,G) \arrow[r,hook] \arrow[u]  & H^1(K_n,G),\arrow[u] 
                        \end{tikzcd}
                \end{equation}
                where the horizontal maps on the left are isomorphisms by \Cref{grothendieck-serre}(a). 
                
                For the injectivity of the horizontal maps on the right side of (\ref{key-diag'}), we argue as follows. Let $L$ be a field containing $k$, $n\geqslant1$ be an integer invertible in $k$, $\cl{L}$ be an algebraic closure of $L$, and $L'_n$ be a field factor of $L_n$. Let $\pi^{1/n}\in \cl{L}$ be such that $(\pi^{1/n})^n=\pi$, and consider the homomorphism $\phi_n \colon L_n\to \cl{L}$ induced by $x\mapsto \pi^{1/n}$. The restriction of $\phi_n$ to $L'_n$ is a field embedding $L'_n\hookrightarrow\cl{L}$ which factors through $L(\pi^{1/n})$. By \Cref{system}(b), we can extend $\pi^{1/n}$ to a system $\set{\pi^{1/m}}_{m\geqslant 1}$ of roots of $\pi$ in $\cl{L}$ such that $(\pi^{1/mm'})^m=\pi^{1/m'}$ for every $m, m'\geqslant 1$.  Letting \[L(\pi^{1/\infty}):=\bigcup_{m\geqslant 1} \, L(\pi^{1/m})\subset \cl{L},\]
                we have the inclusions \[L\subset L'_n\subset L(\pi^{1/\infty}).\]
                Letting $L$ be $\hat{F}$, $K_n'$ or $K_n$, we deduce from \Cref{grothendieck-serre}(c) that the horizontal maps on the right are injective. %(The $\pi^{1/m}$ are not necessarily compatible with the inclusions $K\subset K'$ and $\hat{F}\subset K$, and this compatibility is not needed for the previous argument.) %Moreover, \Cref{claim3} implies that
                %possibly after replacing $K_n''$ by a field extension of degree prime to $q$ and $R'$ by its integral closure in said field extension, 
                %for every $\beta\in H^1(\hat{F},G)$ there exists $n\geqslant 1$ such that the pullback $\beta_{\hat{F}_n}$ is defined over $\hat{O}$. 
                
                Since $\alpha_{K'}$ is defined over $F$, it is defined over $\hat{F}$. We let $\beta \in H^1(\hat{F},G)$ be such that $\beta_{K'}=\alpha_{K'}$. We now observe that $\beta_{\hat{F}(\pi^{1/\infty})}$ is defined over $\hat{O}$. Indeed, if (i') holds then this follows from \cite[Proposition 5.4]{florence2006points}. If (ii') holds then, letting $S$ be a finite subgroup of $G(A)$ as in (ii'), $\beta_{\hat{F}(\pi^{1/\infty})}$ admits reduction of structure to some $\gamma \in H^1(\hat{F}(\pi^{1/\infty}),S)$. Now \Cref{claim3} shows that $\gamma$ is defined over $\hat{O}$, hence so is $\beta_{\hat{F}(\pi^{1/\infty})}$. The field $\hat{F}(\pi^{1/\infty})$ is the increasing union of the $\hat{F}_n$, hence there exists some $n\geqslant 1$ such that $\beta_{\hat{F}_n}$ is defined over $\hat{O}$. Since $\beta_{K''}=\alpha_{K''}$, we have $\beta_{K''_n}=\alpha_{K''_n}$, hence $\alpha_{K''_n}$ is also defined over $\hat{O}$.
                From the commutativity of the right-hand side of (\ref{key-diag'}) and the injectivity of the map $H^1(R',G)\to H^1(K_n'',G)$, we deduce that $\alpha_{R'}$ is defined over $\hat{O}$. The commutativity of the left-hand side of (\ref{key-diag'}) now implies that $\alpha_{K'_0}$ is defined over $F_0$. Since $[K'_0:K_0]$ divides $d$, it is not divisible by $q$. Thus \[\trdeg_{k_0}F_0\geqslant \ed_{k_0,q}\alpha_{K_0}.\]
                Combining this with \eqref{e.ed-alphat'} and \eqref{e.residue-field'}, we conclude that 
                \[\ed_{k,q}\alpha_K= \trdeg_{k}F\geqslant \trdeg_{k_0}F_0 \geqslant\ed_{k_0,q}\alpha_{K_0}.\]
            %Moreover, when $A=k_0[[t]]$ and $G_{A}$ is defined over $k_0$, by \Cref{lem.prel1}(a) and (\ref{e.ed-at-p2}) the reverse inequality is also true, and so $\ed_{k, q}\alpha_K=\ed_{k_0, q}\alpha_{K_0}$.
                %This completes the proof of \Cref{puiseux'}.   
                
                Suppose now that $A=k_0[[t]]$ and that $G_{A}$ is defined over $k_0$. Let $K^{(q)}$ be a $q$-closure of $K$, and let $K_0^{(q)}$ be a $q$-closure of $K_0$ contained in $K^{(q)}$. By (\ref{e.ed-at-p2}), we have $\ed_{k,q}\alpha_K=\ed_{k}\alpha_{K^{(q)}}$ and $\ed_{k_0,q}\alpha_{K_0}=\ed_{k_0}\alpha_{K_0^{(q)}}$. Now \Cref{lem.prel1}(a) implies that $\ed_{k, q}\alpha_K\leqslant\ed_{k_0, q}\alpha_{K_0}$. Therefore $\ed_{k, q}\alpha_K=\ed_{k_0, q}\alpha_{K_0}$, completing the proof of \Cref{puiseux'}.
                \end{proof}
        
        \begin{rmk}
        Recall that in \Cref{cor.not-complete} we proved a version of \Cref{puiseux}, where we did not assume that $A$ and $R$ were complete. One may also prove a variant of \Cref{puiseux'} without assuming that $A$ and $R$ are complete. If $\hat{A}$ denotes the $\mathfrak{m}$-adic completion of $A$, the correct formulation of this variant is obtained by replacing $A$ by $\hat{A}$ in the sentence beginning with ``Furthermore,'' and in assumptions  (i') and (ii'). The proof is identical to the argument used to deduce~\Cref{cor.not-complete} from \Cref{puiseux}, and is left to the reader.
        \end{rmk}

        \begin{rmk}
    Let $G$ be a linear algebraic group over a field $k_0$ of positive characteristic. Then $G_{k_0[[t]]}$ does not satisfy Assumption (ii') of \Cref{puiseux'} in general. Thus~\Cref{puiseux'} is not strong enough to recover~\Cref{thm.map-at-p} 
    in positive characteristic. We will use \Cref{puiseux'} in~\cite{gabber3} to prove prime-to-$q$ analogues of the main results there.
        \end{rmk}

\section*{Appendix B. Cohomological invariants}
\addcontentsline{toc}{section}{Appendix B. Cohomological invariants}
\refstepcounter{section}

Let $k$ be a field of arbitrary characteristic, $G$ be a linear algebraic $k$-group, $X$ be a generically free primitive $G$-variety, and $n\geqslant 0$ be an integer. Consider the following properties of $X$:
    \begin{enumerate}[label=(\roman*)]
        \item $\ed_k(X;G)\geqslant n$.
        \item There exists a cohomological invariant of degree $\geqslant n$ for $G$ which does not vanish on the class of $X$ in $H^1(k(X)^G,G)$.
    \end{enumerate}
It is well known that (ii) implies (i). While (i) and (ii) are not equivalent, they are analogous to each other. 

The purpose of this section is to state and prove an analogue of \Cref{thm.map} in the context of cohomological invariants, 
\Cref{thm.map-cohinv} below. %Note that our proof of~\Cref{thm.map-cohinv} is easier than the argument in Section~\ref{sect.proof}  because it takes advantage of specialization properties of \'etale cohomology, which are not available in the setting of (i). In particular, this proof goes through in arbitrary characteristic.
Our proof of~\Cref{thm.map-cohinv} is easier than the argument in Section~\ref{sect.proof} and goes through in arbitrary characteristic.

Let $\on{Fields}_k$ be the category of fields containing $k$, $\on{Ab}$ be the category of Abelian groups,
$H:\on{Fields}_k\to \on{Ab}$ be a covariant functor. Following A.~S.~Merkurjev, we will require $H$ to satisfy
the following condition:
\begin{equation}\label{condition-star}
    \text{The homomorphism $H(L)\to H(L((t)))$ is injective $\forall$ field extension $L/k$.}
\end{equation}
This is condition $(*)$ of~\cite[p.~108]{garibaldi2003cohomological}. In particular, it is satisfied by the Galois cohomology functor
$K \mapsto H^{d+1}(K, \mathbb Q/ \mathbb Z(d))$ for every $d \geqslant 0$.
Let $G$ be a linear algebraic group over $k$, and let $\on{Inv}(G,H)$ be the set of invariants of $G$ with values in $H$, 
that is, natural transformations $H^1(-,G)\to \on{Forget}\circ H$, where $\on{Forget}:\on{Ab}\to \on{Set}$ is the forgetful functor. 
The set $\on{Inv}(G,H)$ has the structure of an Abelian group.

If $Z$ is a primitive generically free $G$-variety, let $K = k(Z)^G$ and $\tau_Z \colon \mathcal{T} \to \Spec(K)$ be the $G$-torsor constructed in
\Cref{sect.ed-variety}. Let $[Z]\in H^1(K,G)$ denote the class of $\tau_Z$.

\begin{prop} \label{thm.map-cohinv} Let $k$ be a field, let $G$ be a smooth linear algebraic group over $k$, and let $X$, $Y$ be primitive generically free $G$-varieties defined over $k$. Assume that $X$ is smooth, and that there exists a $G$-equivariant rational map $f \colon Y \dasharrow X$. Let $H$ be a functor satisfying \eqref{condition-star}, and let $\alpha,\beta\in \on{Inv}(G,H)$. If $\alpha([X])=\beta([X])$, then $\alpha([Y])=\beta([Y])$.
\end{prop}

\begin{proof}
    Replacing $\alpha$ and $\beta$ by $\alpha-\beta$ and $0$, respectively, we may assume that $\beta=0$.
    
Let $K = k(Y)^G$ and $\tau_Y \colon T \to \Spec(K)$ be the $G$-torsor constructed in \Cref{sect.ed-variety}.
    Set $X_1 := X_{k((t))}$, as in Section~\ref{sect.proof}. 
    Since $\alpha([X]) = 0$, clearly $\alpha([X_1]) = 0$.
    Set $\tau_Y((t)):=\tau_Y \times_{\Spec(K)}\Spec \big( K((t)) \big)$. The same argument as in Section~\ref{sect.proof} tells us that $X_1$ is $\tau_Y((t))$-versal.
    
Let $U \subset X_1$ be a dense open subscheme which is the total space of a $G$-torsor $\pi \colon U \to B$. Recall that $\tau_{X_1}$ 
is the generic fiber of $\pi$.
Since  $X_1$ is $\tau_Y$-versal, there is a morphism of $G$-torsors $\mathcal{T}_Y \to U$ defined over $k$. 
Let $x\in B$ be the image of the induced morphism $\Spec K((t))\to B$. (The point $x$ is not necessarily closed.) 
Since $U$ and $G$ are smooth, so is $B$. It follows that the stalk $\mc{O}_{B,x}$ is a regular local ring. Choose
a local system of parameters $u_1,\ldots,u_n$ in the maximal ideal of $\mc{O}_{B,x}$. For $0\leqslant i\leqslant n$, 
let $x_i$ be the generic point of the variety $B_i \subset B$ defined by $u_1=\cdots=u_i=0$ (we only consider 
the irreducible component containing $x$). In particular, $B_0 = B$.
Then for every $i = 1, \ldots, n$ the point $x_{i+1}$ is regular of codimension $1$ in the closure of $x_i$. 
Moreover, $x_0$ is the generic point of $B$, and $x_n=x$. By assumption, $\alpha([\pi^{-1}(x_0)])=0$. 
Applying~\cite[Part 2, Lemma 3.2]{garibaldi2003cohomological} iteratively, we 
conclude that $\alpha([\pi^{-1}(x_i)])=0$ for all $i$. In particular, $\alpha([\pi^{-1}(x)])=0$. 
Since $\tau_Y((t))$ is a pullback of $\pi^{-1}(x)$, we deduce that $\alpha([Y_{K((t))}])=0$. 
Since $H$ satisfies (\ref{condition-star}), we conclude that $\alpha([Y])=0$, as desired. 
\end{proof}

%%%%%%%%%%%%%%%%%%%%%
% References
%%%%%%%%%%%%%%%%%%%%%

\end{document}